\pdfoutput=1
\newif\ifpersonal
\newif\ifarxiv
\arxivtrue
\RequirePackage[l2tabu, orthodox]{nag} 
\documentclass[12pt,a4paper,reqno]{amsart} 
\usepackage{amsmath,amsthm,amssymb,mathrsfs,mathtools,bm,eucal,tensor} 
\usepackage{microtype,fixltx2e,lmodern} 
\usepackage{enumerate,comment,braket,xspace,tikz-cd,csquotes} 
\usepackage[utf8]{inputenc} 
\usepackage[T1]{fontenc} 
\definecolor{linkcolor}{HTML}{005050}
\usepackage[centering,vscale=0.7,hscale=0.7]{geometry}
\usepackage[hidelinks]{hyperref}
\usepackage[capitalize]{cleveref}

\ifpersonal
\newcommand*{\personal}[1]{\textcolor{blue}{(Personal: #1)}}
\newcommand*{\todo}[1]{\textcolor{red}{(Todo: #1)}}
\else
\newcommand*{\personal}[1]{\ignorespaces}
\newcommand*{\todo}[1]{\ignorespaces}
\fi

\theoremstyle{plain}
\newtheorem{thm*}{Theorem}
\newtheorem{thm}{Theorem}[section]
\newtheorem{lem}[thm]{Lemma}
\newtheorem{prop}[thm]{Proposition}
\newtheorem{conj}[thm]{Conjecture}
\newtheorem{cor*}[thm*]{Corollary}
\newtheorem{cor}[thm]{Corollary}

\theoremstyle{definition}
\newtheorem{defin}[thm]{Definition}
\newtheorem{defin*}[thm*]{Definition}

\newtheorem{eg*}[thm*]{Example}

\newtheorem{rem*}[thm*]{Remark}
\newtheorem{rem}[thm]{Remark}
\newtheorem{construction}[thm]{Construction}
\theoremstyle{remark}
\numberwithin{equation}{section}

\newcommand{\cA}{\mathcal A}
\newcommand{\cB}{\mathcal B}
\newcommand{\cC}{\mathcal C}
\newcommand{\cE}{\mathcal E}
\newcommand{\cH}{\mathcal H}
\newcommand{\cJ}{\mathcal J}
\newcommand{\cO}{\mathcal O}
\newcommand{\cT}{\mathcal T}
\newcommand{\cS}{\mathcal S}
\newcommand{\cX}{\mathcal X}
\newcommand{\fX}{\mathfrak X}
\newcommand{\cY}{\mathcal Y}

\newcommand{\cF}{\mathcal F}
\newcommand{\cG}{\mathcal G}

\newcommand{\cD}{\mathcal D}

\newcommand{\PSh}{\mathrm{PSh}}
\newcommand{\Sh}{\mathrm{Sh}}

\newcommand{\inv}{^{-1}}

\newcommand{\Cech}{\check{\mathcal C}}

\newcommand{\canal}{$\mathbb C$-analytic\xspace}

\DeclareMathAlphabet{\mathpzc}{OT1}{pzc}{m}{it}

\newcommand{\sSet}{\mathrm{sSet}}
\newcommand{\rSet}{\mathrm{Set}}
\newcommand{\Ab}{\mathrm{Ab}}

\newcommand{\Mod}{\mathrm{Mod}}

\newcommand{\Coh}{\mathrm{Coh}}

\newcommand{\trunc}{\mathrm{t}_0}
\newcommand{\et}{\mathrm{\acute{e}t}}
\newcommand{\Jac}{\mathrm{Jac}}
\newcommand{\fib}{\mathrm{fib}}
\newcommand{\Cat}{\mathcal C \mathrm{at}}
\newcommand{\PStab}{\mathrm{PStab}}
\newcommand{\Stab}{\mathrm{Stab}}

\newcommand{\cTdisc}{\cT_{\mathrm{disc}}}

\newcommand{\cTzar}{\cT_{\mathrm{Zar}}}

\newcommand{\cTet}{\cT_{\mathrm{\acute{e}t}}}

\newcommand{\Tan}{\cT_{\mathrm{an}}}

\newcommand{\cotimes}{\widehat{\otimes}}
\newcommand{\Strloc}{\mathrm{Str}^\mathrm{loc}}

\newcommand{\alg}{^\mathrm{alg}}
\newcommand{\an}{^\mathrm{an}}

\newcommand{\Str}{\mathrm{Str}}

\newcommand{\CRing}{\mathrm{CRing}}
\newcommand{\DM}{Deligne-Mumford\xspace}

\DeclareMathOperator{\Hom}{Hom}

\DeclareMathOperator{\Map}{Map}
\DeclareMathOperator{\Fun}{Fun}

\DeclareMathOperator{\Spec}{Spec}
\DeclareMathOperator{\Sp}{Sp}

\DeclareMathOperator*{\colim}{colim}

\begin{document}

\title{Derived $\mathbb C$-analytic geometry II: square-zero extensions}

\author{Mauro PORTA}
\address{Mauro PORTA, Institut de Math\'ematiques de Jussieu, CNRS-UMR 7586, Case 7012, Universit\'e Paris Diderot - Paris 7, B\^atiment Sophie Germain 75205 Paris Cedex 13 France}
\email{mauro.porta@imj-prg.fr}
\date{June 29, 2014 (Revised on \today)}


\begin{abstract}
	We continue the explorations of derived \canal geometry started in \cite{DAG-IX} and \cite{Porta_GAGA_2015}.
	We describe the category of $\cO_X$-modules over a derived \canal space $X$ as the stabilization of a suitable category of analytic algebras over $\cO_X$.
	Finally, we apply this description to introduce the notion of analytic square-zero extension and prove a fundamental structure theorem for them.
\end{abstract}

\maketitle

\personal{PERSONAL COMMENTS ARE SHOWN!!!}

\tableofcontents

\section*{Introduction}

We referred to this article in \cite{Porta_GAGA_2015} with the title ``Derived \canal geometry II: deformation theory''.
We finally decided to split the content originally meant to appear here into two different articles: the current one and \cite{Porta_Cotangent_2015} that will soon be available.
We did so because the article was growing longer than expected. The material has been divided as follows: in the current article we discuss to a great length the notion of analytic module over a derived \canal space $(X, \cO_X)$. We provide an alternative description of $\cO_X\alg \textrm{-} \Mod$ which is particularly suited for the study of infinitesimal deformation theory in this derived \canal setting.
As a first application, we introduce the notion of analytic square-zero extensions and we prove an important structure theorem for them (\cref{cor:structure_theorem_square_zero}).

On the other side, the analytic cotangent complex will be discussed thoroughly in \cite{Porta_Cotangent_2015}. For this reason, we felt the need to change the title of the current article to ``Derived \canal geometry II: square-zero extensions''.

\medskip

\paragraph{\textbf{Overview and organization of the paper}}

In this article we continue the explorations of derived \canal geometry started in \cite{Porta_GAGA_2015}.
We briefly recall that there we chose to adopt the foundations proposed by J.\ Lurie in \cite{DAG-IX}, and in this paper we will continue to do so.
We refer to the introduction of \cite{Porta_GAGA_2015} for an expository account of the main ideas involved in this approach.
We can summarize the main results we previously obtained by saying that we introduced a notion of coherent sheaf over a derived \canal space and that we proved it to be solid by showing that it leads to a version of both Grauert's proper direct image theorem and the two GAGA theorems.

We now continue in the very natural direction of (infinitesimal) deformation theory. One of the attracting features of derived algebraic geometry is that it provides a very powerful framework to deal with deformation theory. In a sense, one could even bring himself to say that the objects of study in derived algebraic geometry can be naturally decomposed into two ``orthogonal'' directions: a stacky (underived) part and a purely derived part; and the latter is completely determined by (infinitesimal) deformation theory. There are indeed many ways to make this idea precise, or, better, there are many instances of this general principle.
Here we collect some of the ones we consider most significant:
\begin{enumerate}
	\item if $X$ is a derived \DM stack, $\trunc(X)$ denotes its truncation, and $X_\et$, $\trunc(X)_\et$ denote their small \'etale sites, then pullback along the closed immersion $\trunc(X) \to X$ induce an equivalence
	\[ X_\et \simeq \trunc(X)_\et . \]
	One of the reasons this is a very important fact is that it implies that the small \'etale topos of $X$ is $n$-localic for some $n$ and that therefore it behaves much like it was hypercomplete (despite not being hypercomplete). We refer e.g.\ to \cite[Proposition 2.3]{Porta_Comparison_2015} for an application of this result.
	\item If $k$ is a discrete commutative ring and $X = \Spec(A)$ is an affine derived scheme over $\Spec(k)$ (e.g.\ the reader can think of $A$ as a simplicial commutative $k$-algebra), we can always represent $A$ as the inverse limit of its Postnikov tower:
	\[ A \simeq \lim \tau_{\le n}(A) \]
	where the truncation is taken in the $\infty$-category $\mathrm{CAlg}_k$.
	This decomposition is very useful because we can say a lot about the structure of the morphisms $\tau_{\le n}(A) \to \tau_{\le n-1}(A)$.
	First of all, they all correspond to closed immersion.
	More strikingly, they always are \emph{square-zero extensions}. The structure theorem for square-zero extensions (see e.g.\ \cite[7.4.1.26]{Lurie_Higher_algebra} or, for a more elementary formulation, \cite{Porta_Vezzosi_Square_zero}) implies therefore that one can always choose a $k$-linear derivation $d \colon \tau_{\le n-1}(A) \to \tau_{\le n - 1}(A) \oplus \pi_n(A)[n+1]$ in such a way that
	\[ \begin{tikzcd}
	\tau_{\le n}(A) \arrow{r} \arrow{d} & \tau_{\le n - 1}(A) \arrow{d}{d} \\
	\tau_{\le n - 1}(A) \arrow{r}{d_0} & \tau_{\le n - 1}(A) \oplus \pi_n(A)[n+1]
	\end{tikzcd} \]
	is a pullback square. Here $\tau_{\le n - 1}(A) \oplus \pi_n(A)[n+1]$ denotes the split square-zero extension of $\tau_{\le n - 1}(A)$ associated to the module $\pi_n(A)[n+1]$ and $d_0$ denotes the trivial derivation.
	\item Finally, there is Lurie's representability theorem, which is perhaps the most important manifestation of the guiding principle we are discussing. Roughly speaking, this theorem says that a derived stack $X$ is a geometric stack (and therefore in particular it admits an atlas by affine derived schemes) if and only if its truncation $t_0(X)$ is a geometric stack on its own and $X$ has a good infinitesimal deformation theory. We refer to \cite[Appendix C]{HAG-II} for this formulation and to \cite{DAG-XIV} for a more leisurely exposition.
\end{enumerate}

These three techniques combined together form a very powerful and useful toolkit that allows to prove statements concerning derived algebraic geometry by splitting them into an underived statement and a problem in infinitesimal deformation theory.
At that point, one can deal with the underived part with more classical techniques; on the other side, one can heavily draw upon the cotangent complex formalism to understand the infinitesimal deformation part.

An important part of the project started in \cite{Porta_Yu_Higher_analytic_stacks_2014,Porta_GAGA_2015} is to show that a similar toolkit can be developed also in derived \canal geometry.
In fact, we already dealt with the first and most unrefined of the previous techniques in \cite[Proposition 3.15]{Porta_GAGA_2015}, and we applied it to deduce all the major results of that article precisely by reducing to their underived analogues, which had been previously dealt with in \cite{Porta_Yu_Higher_analytic_stacks_2014}.
In this article, we will focus on the \canal versions of point (2), and we will come back to the third point in \cite{Porta_Cotangent_2015}, where we will also discuss to a great extent the notion of analytic cotangent complex.

\medskip

\paragraph{\textbf{Modules as tangent categories.}}

Let us start by recalling that $\Tan$ denotes the pregeometry consisting of open subsets of $\mathbb C^n$, where admissible morphisms are precisely open immersions and we consider the usual analytic topology.
In first approximation, a derived \canal space is a pair $(X, \cO_X)$, where $X$ is a topological space and $\cO_X$ is a $\Tan$-structure $\cO_X \colon \Tan \to \Sh(X)$, whose underlying locally ringed space is a \canal space.
To explain what is the underlying locally ringed space, we recall from the introduction of \cite{Porta_GAGA_2015} that $\cO_X$ can be roughly thought as a sheaf of simplicial commutative $\mathbb C$-algebras satisfying an axiomatic version of the holomorphic functional calculus typical of Banach algebras.
We can therefore forget the extra structure giving the holomorphic functional calculus to remain with a sheaf $\cO_X\alg$ with values in the $\infty$-category $\mathrm{CAlg}_{\mathbb C}$ of simplicial commutative $\mathbb C$-algebras, and at this point we can further apply the $\pi_0$ to get a sheaf of classical $\mathbb C$-algebras whose stalks are local rings.
Therefore the requirement is that $(X, \pi_0(\cO_X\alg))$ is a \canal space.

In \cite[Definition 4.1]{Porta_GAGA_2015} we defined a sheaf of $\cO_X$-modules on $(X, \cO_X)$ to be a sheaf of $\cO_X\alg$-modules.
This is in line with what it is usually done for $\mathcal C^\infty$-rings (see e.g.\ \cite{Joyce_Cinfty_rings}), but yet it is awkward because this definition forgets completely the interesting analytic part of $\cO_X$.
There is a more natural definition that it is possible to consider.
Namely, Quillen remarked that if $A$ is a commutative ring then the category of $A$-modules is equivalent to $\Ab(\CRing_{/A})$.
This definition uses only the internal structure of the category of commutative rings, and therefore it is suitable to be extended to more general algebraic structures.
Now, in higher category theory the role of internal abelian groups is played by the spectrum objects, and therefore it is tempting to define the category of $\cO_X$-modules as
\[ \cO_X \textrm{-} \Mod \coloneqq \Sp(\Strloc_{\Tan}(\cX)_{/\cO_X}) \]
Since we were able to prove GAGA theorems using $\cO_X\alg \textrm{-} \Mod$ as category of $\cO_X$-modules, it becomes interesting to try to understand the precise relation between these two $\infty$-categories.

It is not difficult to construct a comparison functor
\[ \Phi \colon \Sp(\Strloc_{\Tan}(\cX)_{/\cO_X}) \to \cO_X\alg \textrm{-} \Mod \]
We then prove:

\begin{thm*}[\cref{thm:equivalence_of_modules}]
	Suppose that $(\cX, \cO_\cX)$ is a derived \canal space.
	Then the comparison functor $\Phi$ is an equivalence of $\infty$-categories.
\end{thm*}

The whole \cref{sec:sheaves_in_canal_geometry} is devoted to discuss this theorem.
The proof is actually quite long and it goes through several reduction steps. We refer to \cref{subsec:modules_theorem_exposition} for a detailed outline of the strategy and of the main ideas involved.
We remark that we stated this theorem only for derived \canal spaces and not for general $\Tan$-structured topoi.
We strongly believe the latter statement is true, but we are not yet able to prove it.
We discuss at length this conjecture in \cref{subsec:conjectures}.

\medskip

\paragraph{\textbf{The structure theorem for analytic square-zero extensions}}

So far we presented \cref{thm:equivalence_of_modules} from a purely philosophical point of view.
However, as we already anticipated, this is far from being an idle question. Indeed, the new definition we are taking into consideration has the advantage of leading to a very natural theory of square-zero extensions.
To explain why it is so, let us start by briefly recalling this notion in the more classical setting of connective $\mathbb E_\infty$-rings.
The category of $\mathbb E_\infty$-rings is defined as the category of algebras for the (monochromatic) $\infty$-operad $\mathbb E_\infty$ \cite[7.1.0.1]{Lurie_Higher_algebra}, and modules for a given $\mathbb E_\infty$-ring $R$ can be defined as algebras for the (bichromatic) operad $\mathcal{LM}^\otimes$.
If we want to deal with square-zero extensions, all we have to do in virtue of the structure theorem \cite[7.4.1.26]{Lurie_Higher_algebra} is to have a good understanding of split square-zero extension (see e.g.\ \cite[§1.1]{DAG-X} for an axiomatic treatment of the notion of artinian object that implicitly relies on this result).
Let $R$ be an $\mathbb E_\infty$-ring and let $M$ be an $R$-module.
The split square-zero extension of $R$ by $M$ should be an $\mathbb E_\infty$-ring structure on $R \oplus M$ such that the multiplication map
\[ M \otimes M \to (R \oplus M) \otimes (R \oplus M) \xrightarrow{m} R \oplus M \]
is nullhomotopic.
However, as usual in higher category theory, we cannot define such a multiplicative structure by writing down equations, and therefore a cleverer way has to be found.
We can summarize J.\ Lurie's approach in \cite[§7.3.4]{Lurie_Higher_algebra} as follows.
Instead of directly constructing the object $R \oplus M$, we consider the would-be functor
\[ R \textrm{-} \Mod \to \mathrm{CAlg}_{/R} \]
informally described by $M \mapsto R \oplus M$ (here $\mathrm{CAlg} = \mathrm{CAlg}(\Sp)$ denotes the $\infty$-category of connective $\mathbb E_\infty$-ring spectra).
This functor should preserve limits and therefore it should factor as
\[ R \textrm{-} \Mod \xrightarrow{G} \Sp(\mathrm{CAlg}_{/R}) \xrightarrow{\Omega^\infty} \mathrm{CAlg}_{/R} \]
At this point, what is really done is to construct explicitly a functor $F \colon \Sp(\mathrm{CAlg}_{/R}) \to R \textrm{-} \Mod$ and use some Goodwillie's calculus to prove that $F$ is actually an equivalence (see \cite[7.3.4.14]{Lurie_Higher_algebra}). This permits to define $G$ as a quasi-inverse of $F$ and our desired split square-zero extension functor is now defined as $\Omega^\infty \circ G$.

In the analytic setting, \cref{thm:equivalence_of_modules} plays the role of \cite[7.3.4.14]{Lurie_Higher_algebra}.
More explicitly, if $(\cX, \cO_\cX)$ is a derived \canal space and $M \in \cO_\cX\alg \textrm{-} \Mod$, we will define the split square-zero extension of $\cO_\cX$ by $M$ to be $\Omega^\infty_{\mathrm{an}}(M)$, where
\[ \Omega^\infty_{\mathrm{an}} \colon \Sp(\Strloc_{\Tan}(\cX)_{/\cO}) \to \Strloc_{\Tan}(\cX)_{\cO // \cO} \]
is the usual forgetful functor from the stabilization of an $\infty$-category $\cC$ to the $\infty$-category $\cC$ itself, and we review $M$ as an element of $\Sp(\Strloc_{\Tan}(\cX)_{/\cO})$ using \cref{thm:equivalence_of_modules}.
This leads us to a very natural notion of analytic derivation, which will simply be a section of the projection map $\Omega^\infty_{\mathrm{an}}(M) \to \cO$ in the $\infty$-category $\Strloc_{\Tan}(\cX)_{/ \cO}$.
We will discuss better this idea in the next paragraph, while dealing with the notion of cotangent complex.
For the moment, we are interested in this because it allows to introduce the notion of square-zero extension in derived \canal geometry:

\begin{defin*}[\cref{def:analytic_square_zero_extension}]
	Let $\cX$ be an $\infty$-topos and let $f \colon \cO' \to \cO$ be a morphism in $\Strloc_{\Tan}(\cX)$.
	We will say that $f$ is \emph{an analytic square-zero extension} if there exists an analytic $\cO$-module $M$ and an $\cO'$-linear derivation $d \colon \cO \to \Omega^\infty_{\mathrm{an}}(M)$ such that the square
	\[ \begin{tikzcd}
	\cO' \arrow{r} \arrow{d}{f} & \cO \arrow{d}{d} \\
	\cO \arrow{r}{d_0} & \Omega^\infty_{\mathrm{an}}(M)
	\end{tikzcd} \]
	is a pullback in $\Strloc_{\Tan}(\cX)_{/\cO}$.
	In this case, we will say that $f$ is a square-zero extension of $\cO$ by $M[-1]$.
\end{defin*}

This notion of square-zero extension would be pretty much useless if it didn't come with a recognizing criterion that it allows in practice to decide whether a given map is a square-zero extension or not.
This is the goal of the structure theorem of square-zero extensions.
We introduce the following algebraic notion:

\begin{defin*}[{\cref{def:analytic_small_extension}}]
	Let $\cX$ be an $\infty$-topos and let $f \colon A \to B$ be a morphism in $\Strloc_{\Tan}(\cX)$.
	We will say that $f$ is an $n$-small extension if the morphism $f\alg \colon A\alg \to B\alg$ is an $n$-small extension in the sense of \cite[Definition 7.4.1.18]{Lurie_Higher_algebra}, that is if the following two conditions are satisfied:
	\begin{enumerate}
		\item The fiber $\mathrm{fib}(f\alg)$ is $n$-connective and $(2n)$-truncated, and
		\item the multiplication map $\mathrm{fib}(f\alg) \otimes_{A\alg} \mathrm{fib}(f\alg) \to \mathrm{fib}(f\alg)$ is nullhomotopic.
	\end{enumerate}
\end{defin*}

These two conditions don't concern the analytic structure on the morphism $f$ at all, and moreover they can be tested on the homotopy groups.
Therefore, it is often a fairly task to check in practice whether a given analytic morphism is an $n$-small extension.
The structure theorem for square-zero extension says that the two notions coincide:

\begin{thm*}[{\cref{cor:structure_theorem_square_zero}}]
	Let $(\cX, \cO_\cX)$ be a derived \canal space.
	Let $\cO \in \Str_{\Tan}(\cX)$ be any other analytic structure on $\cX$ and let $f \colon \cO \to \cO_\cX$ be a local morphism.
	Fix furthermore a non-negative integer $n$.
	Then the following conditions are equivalent:
	\begin{enumerate}
		\item $f$ is an analytic square-zero extension by $M[-1]$ and $\Omega^\infty_{\mathrm{an}}(M)$ is $n$-connective and $(2n)$-truncated;
		\item $f$ is an $n$-small extension.
	\end{enumerate}
\end{thm*}

As an immediate consequence, we obtain:

\begin{cor*}[\cref{cor:truncations_are_analytic_square_zero}]
	Let $(\cX, \cO_\cX)$ be a derived \canal space.
	For every non-negative integer $n$, the natural map $\tau_{\le n} \cO_\cX \to \tau_{\le n - 1} \cO_\cX$ is a square-zero extension.
\end{cor*}

\medskip

\paragraph{\textbf{The appendixes}}

Even though the results of Appendixes \ref{sec:relative_analytification} and \ref{sec:left_adjointable} are mainly needed as auxiliary tools in \cref{sec:sheaves_in_canal_geometry}, the results we obtain there are somehow interesting on their own.
In \cref{sec:relative_analytification} we prove a flatness result analogous to \cite[Theorem 6.19]{Porta_GAGA_2015} in what should be called a ``relative setting'':

\begin{thm*}[\cref{prop:relative_flatness}]
	Let $\cO \in \Str_{\Tan}(\cS)$  be the germ of a derived \canal space.
	There is an adjunction
	\[ \overline{\Psi}_\cO \colon \Strloc_{\cTzar}(\cS)_{\cO\alg /} \rightleftarrows \Strloc_{\Tan}(\cS)_{\cO /} \colon \overline{\Phi}_\cO \]
	where $\overline{\Phi}_{\cO}$ is given by precomposition along the morphism of pregeometries $\varphi \colon \cTzar \to \Tan$.
	Moreover, if $A \in \Strloc_{\cTzar}(\cS)_{\cO\alg /}$ the canonical map $A \to \overline{\Phi}_{\cO}(\overline{\Psi}_\cO(A))$ is flat (in the derived sense).
\end{thm*}

The functor $\overline{\Psi}_\cO$ should really be thought of as the analytification functor relative to the analytic base $(\cS, \cO)$, much like in the sense of M.\ Hakim's thesis \cite{Hakim_Topos_1972}. It is possible to greatly expand this observation, and we will come back to this subject in a subsequent work.

In \cref{sec:left_adjointable} we take advantage of a particularly simple description of left adjointable squares introduced by M.\ Hopkins and J.\ Lurie in \cite{Lurie_Ambidexterity} to prove the stability of these objects under a number of operations.
Namely, we prove that they can be composed horizontally and vertically and, most importantly, that stabilization preserves (good) left adjointable squares:

\begin{thm*}[\cref{prop:stabilization_of_left_adjointable_squares}]
	Let
	\[ \begin{tikzcd}
	\cB_1 & \cA_1 \arrow{l}[swap]{G_1} \\
	\cB_0 \arrow{u}{P} & \cA_0 \arrow{l}[swap]{G_0} \arrow{u}[swap]{Q}
	\end{tikzcd} \]
	be a left adjointable square of pointed $\infty$-categories with finite limits (where $G_0$ and $G_1$ are the right adjoints).
	Suppose furthermore that $P$ and $Q$ are left exact functors and that they commute with sequential colimits.
	Then the induced square
	\[ \begin{tikzcd}
	\Sp(\cB_1) & \Sp(\cA_1) \arrow{l}[swap]{\partial G_1} \\
	\Sp(\cB_0) \arrow{u}{\partial P} & \Sp(\cA_0) \arrow{l}[swap]{\partial G_0} \arrow{u}[swap]{\partial Q}
	\end{tikzcd} \]
	is left adjointable as well. Here the symbol $\partial$ stands for the first Goodwillie derivative.
\end{thm*}

\medskip

\paragraph{\textbf{Conventions}}

We will work freely with the language of $(\infty,1)$-categories.
We will call them simply $\infty$-categories and our basic reference on the subject is \cite{HTT}.
The notation $\cS$ will be reserved for the $\infty$-categories of spaces.
Whenever categorical constructions are used (such as limits, colimits etc.), we mean the corresponding $\infty$-categorical notion.

In \cite{HTT} and more generally in the DAG series, whenever $\cC$ is a $1$-category the notation $\mathrm N(\cC)$ denotes $\cC$ reviewed (trivially) as an $\infty$-category. This notation stands for the nerve of the category $\cC$ (and this is because an $\infty$-category in \cite{HTT} is defined to be a quasicategory, that is a simplicial set with special lifting properties). In this note, we will systematically suppress this notation, and we encourage the reader to think to $\infty$-categories as model-independently as possible. For this reason, if $k$ is a (discrete) commutative ring we chose to denote by $\mathrm{CRing}_k$ the $1$-category of discrete $k$-algebras and by $\mathrm{CAlg}_k$ the $\infty$-category underlying the category of simplicial commutative $k$-algebras.
Concerning $\infty$-topoi, we choose to denote geometric morphisms by $f\inv \colon \cX \rightleftarrows \cY \colon f_*$. We warn the reader that in \cite{HTT} the notation $f^*$ is used instead of $f\inv$. We reserve $f^*$ to denote the pullback in the category of $\cO_X$-modules.

We will use the language of pregeometries introduced in \cite{DAG-V}. We refer to the introduction of \cite{Porta_GAGA_2015} for an expository account, as well as for the definition of the four main pregeometries we will work with: $\cTdisc$, $\cTzar$, $\cTet$ and $\Tan$.

\medskip

\paragraph{\textbf{Acknowledgments}}

\cref{thm:equivalence_of_modules} occupied my mind for many months.
The proof uses a number of different ideas that arrived to me through several channels.
First of all, I want to emphasize my intellectual debt to J.\ Lurie: the idea of using the Goodwillie calculus was derived from \cite{Lurie_Higher_algebra}.
Next, I want to express my deepest gratitude to Gabriele Vezzosi and Marco Robalo, for listening to a huge number of never-working strategies for quite a long time, as well as for suggesting many alternative paths.
Finally, I want to warmly thank Tony Yue Yu for teaching (and periodically reminding) me the importance of the relative analytification, which revealed itself the very last piece of the puzzle in the proof of \cref{thm:equivalence_of_modules}.

\section{Sheaves of modules in derived \canal geometry} \label{sec:sheaves_in_canal_geometry}

\subsection{Statements of the main results} \label{subsec:modules_theorem_exposition}

Let $X = (\cX, \cO_\cX)$ be a derived \canal space.
In \cite[§4]{Porta_GAGA_2015} we defined the $\infty$-category $\Coh(X)$ as a full subcategory of $\cO_\cX\alg \textrm{-} \Mod$.
Let us recall that $\cO_\cX\alg$ is the $\cTdisc$-structure on $\cX$ obtained by composing $\cO_\cX$ with the morphism of pregeometries $\varphi \colon \cTdisc \to \Tan$. Since $\cTdisc$-structures are precisely sheaves of connective $\mathbb E_\infty$-rings which are $\mathrm H \mathbb C$-algebras, we will think of $\cO_\cX\alg$ as a sheaf with values in the $\infty$-category $\mathrm{CAlg}_{\mathbb C}$.

As we explained in the introduction, there is at least another reasonable definition of the $\infty$-category of $\cO_\cX$-modules:

\begin{defin} \label{def:analytic_modules}
	Let $X = (\cX, \cO_\cX)$ be a $\Tan$-structured topos.
	The $\infty$-category of \emph{analytic $\cO_\cX$-modules} is the stabilization of $\Strloc_{\Tan}(\cX)_{/\cO}$, i.e.\ $\Sp(\Strloc_{\Tan}(\cX)_{/\cO})$.
\end{defin}

The main goal of this section is to compare the $\infty$-category of analytic $\cO_\cX$-modules with $\cO_\cX\alg \textrm{-} \Mod$.
In virtue of \cite[7.3.4.14]{Lurie_Higher_algebra}, we can further describe this $\infty$-category as
\[ \Sp(\Strloc_{\cTdisc}(\cX)_{/\cO\alg}) \]
Using the morphism of pregeometries $\varphi \colon \cTdisc \to \Tan$ and \cite[Corollary 1.18]{Porta_GAGA_2015} we immediately obtain a comparison functor
\begin{equation} \label{eq:comparison_functor_modules}
\Phi_\cO \colon \cO \textrm{-} \Mod_{\Tan} = \Sp(\Strloc_{\Tan}(\cX)_{/\cO}) \to \Sp(\Strloc_{\cTdisc}(\cX)_{/\cO\alg}) = \cO\alg \textrm{-} \Mod
\end{equation}
We can now state the main result of this first section:

\begin{thm} \label{thm:equivalence_of_modules}
	Let $X = (\cX, \cO_\cX)$ be a derived \canal space.
	Then the comparison functor \eqref{eq:comparison_functor_modules} is an equivalence of $\infty$-categories.
\end{thm}

Observe that the functor $\Phi_\cO$ commutes with limits and filtered colimits. Therefore, it admits both a left adjoint $\Psi_\cO$ and a right adjoint $\Xi_\cO$.
Moreover, the functor $\overline{\Phi}_\cO$ is conservative (see \cite[Proposition 11.9]{DAG-IX}).
It follows that $\Phi_\cO$ is conservative as well. Therefore, \cref{thm:equivalence_of_modules} is equivalent to prove that the functor $\Psi_\cO$ is fully faithful.
Our strategy to prove this statement is the following: we first proceed to a preliminary study of the properties of the category $\Sp(\Strloc_{\Tan}(\cX)_{/\cO_\cX})$; in particular, we aim at endowing such category with a $t$-structure and to study its behavior with respect to the operation of forgetting along the morphism of pregeometries $\cTdisc \to \Tan$.
Next, we argue that it is enough to prove \cref{thm:equivalence_of_modules} for $(\cS, \cO)$, where $\cO$ is the stalk of the structure sheaf of some derived \canal space at a geometric point.

We then use the flatness result of the relative analytification discussed in \cref{sec:relative_analytification} to further reduce to the case where $\cO$ is discrete.
After this final d\'evissage step, we can finally deal with the problem in a very explicit way.
This will complete the proof of \cref{thm:equivalence_of_modules}.
We will end the section by discussing a possible generalization of \cref{thm:equivalence_of_modules} to the case of a generic $\Tan$-structured topos.

\subsection{Modules for a pregeometry} \label{subsec:modules_pregeometry}

As we just explained, our first goal is to endow the $\infty$-category $\Sp(\Strloc_{\Tan}(\cX)_{/\cO_\cX})$ with a reasonable $t$-structure.
Since no particular feature of the pregeometry $\Tan$ is needed in deducing the main results, we will be working with a rather general pregeometry $\cT$.
The analogue of \cref{def:analytic_modules} in this setting is the following:

\begin{defin} \label{def:modules_for_T}
	Let $\cT$ be a pregeometry and let $\cX$ be an $\infty$-topos.
	Let $\cO \in \Str_\cT(\cX)$. The $\infty$-category of \emph{$\cO$-modules (for the pregeometry $\cT$)} is defined to be
	\[ \cO \textrm{-} \Mod_{\cT} \coloneqq \Sp(\Strloc_{\cT}(\cX)_{/\cO}) \]
	When the pregeometry $\cT$ is clear from the context, we will simply denote $\cO \textrm{-} \Mod_{\cT}$ by $\cO \textrm{-} \Mod$.
\end{defin}

Let us start by recalling that in \cite{Porta_GAGA_2015} the notion of weak Morita equivalence allowed us to prove the following result:

\begin{prop} \label{prop:O_modules_presentable}
	Let $\cT$ be a pregeometry and let $\cX$ be an $\infty$-topos.
	For every $\cO \in \Str_{\cT}(\cX)$ the $\infty$-category $\cO \textrm{-} \Mod$ is presentable and stable.
\end{prop}

\begin{proof}
	The category $\cO \textrm{-} \Mod$ is the stabilization of $\Strloc_{\cT}(\cX)_{/\cO}$ and therefore it is stable, see \cite[1.4.2.17]{Lurie_Higher_algebra}.
	The presentability has instead been proved in \cite[Corollary 1.16]{Porta_GAGA_2015}.
\end{proof}

The previous proposition together with \cite[Proposition 1.4.4.11]{Lurie_Higher_algebra} allow to put on $\cO \textrm{-} \Mod$ the $t$-structure we are interested in quite easily.
Nevertheless, to prove certain special features of this $t$-structure, some additional work is needed.
We will begin by discussing the content of \cite[Proposition 1.7]{DAG-VII} in a slightly different setting.
The next couple of lemmas will prove quite useful in what follows.

\begin{lem}[{Cf.\ \cite[Lemma 3.3.4]{DAG-V}}] \label{lem:truncation_and_pullbacks}
	Let $\cX$ be an $\infty$-topos and suppose given a pullback square
	\[ \begin{tikzcd}
	K' \arrow{r} \arrow{d} & K \arrow{d}{p} \\
	L' \arrow{r}{q} & L
	\end{tikzcd} \]
	If $K$ and $L$ are $n$-truncated, then the functor $\tau_{\le n} \colon \cX \to \tau_{\le n} \cX$ commutes with this pullback.
\end{lem}

\begin{proof}
	Since $L$ is $n$-truncated, the morphism $L' \to L$ factors as $L' \to \tau_{\le n} L' \to L$.
	Since $L$ is $n$-truncated, we can think of this as the $n$-truncation in $\cX_{/L}$ (see \cite[5.5.6.14]{HTT}).
	Since the pullback functor $- \times_L K \colon \cX_{/L} \to \cX_{/K}$ is both a left and a right adjoint, it commutes with the truncation functors (see \cite[5.5.6.28]{HTT}).
	Therefore
	\[ \tau_{\le n} L' \times_L K = \tau_{\le n}^{\cX_{/L}}(L') \times_L K \simeq \tau_{\le n}^{\cX_{/K}}(L' \times_L K ) \simeq \tau_{\le n}^{\cX_{/K}}(K') \]
	Since $K$ is $n$-truncated, we see that $\tau_{\le n}^{\cX_{/K}}(K') \simeq \tau_{\le n} K'$ by \cite[5.5.6.14]{HTT} again.
	Therefore the diagram
	\[ \begin{tikzcd}
	\tau_{\le n} K' \arrow{r} \arrow{d} & K \arrow{d} \\
	\tau_{\le n} L' \arrow{r} & L
	\end{tikzcd} \]
	is a pullback square. As $\tau_{\le n} K \simeq K$ and $\tau_{\le n} L \simeq L$, the proof is complete.
\end{proof}

\begin{lem} \label{lem:truncation_and_overcategories}
	Let $\cX$ be an $\infty$-topos and let $X \in \cX$ be an object.
	\begin{enumerate}
		\item The functor $j\inv \colon \cX_{/ \tau_{\le n} X} \to \cX_{/ X}$ given by the pullback along $X \to \tau_{\le n} X$ induces a fully faithful functor
		\[ \sigma \colon \tau_{\le n}( \cX_{/\tau_{\le n} X} ) \to \tau_{\le n} (\cX_{/X} ) \]
		\item Suppose furthermore that there exists an integer $m$ such that $X$ is $m$-truncated or that $\cX$ is hypercomplete. Then the essential image of $\sigma$ is given by $n$-truncated maps $Y \to X$ inducing isomorphisms $\pi_{n+1}(Y,x) \to \pi_{n+1}(X,x)$ for every basepoint $x$.
		\item In particular, if $X$ is $m$-truncated for some $m \le n$, then the truncation functor $\tau_{\le n} \colon \cX \to \tau_{\le n} \cX$ induces an equivalence $\tau_{\le n}(\cX_{/X}) \to \tau_{\le n}(\cX)_{/X}$.
	\end{enumerate}
\end{lem}

\begin{proof}
	The pullback functor $j\inv \colon \cX_{/\tau_{\le n} X} \to \cX_{/X}$ is both a left and a right adjoint.
	Therefore \cite[Proposition 5.5.6.16]{HTT} shows that the composite
	\[ \tau_{\le n}(\cX_{/\tau_{\le n} X}) \to \cX_{/\tau_{\le n} X} \xrightarrow{j\inv} \cX_{/X} \]
	factors as $\sigma \colon \tau_{\le n}( \cX_{/\tau_{\le n} X} ) \to \tau_{\le n} (\cX_{/X} )$ in such a way that the diagram
	\[ \begin{tikzcd}
	\cX_{/\tau_{\le n} X} \arrow{r}{j\inv} & \cX_{/ X} \\
	\tau_{\le n} (\cX_{/\tau_{\le n} X}) \arrow{u}{i_n} \arrow{r}{\sigma} & \tau_{\le n} (\cX_{/X}) \arrow{u}{i}
	\end{tikzcd} \]
	commutes, where the vertical arrows are the inclusion functors.
	This diagram shows that $\sigma$ commutes with limits and $\kappa$-filtered colimits for $\kappa$ large enough.
	Therefore it admits a left adjoint, which we will denote $\pi$. We observe that $j\inv$ is right adjoint to the \emph{forgetful} functor $j_! \colon \cX_{/X} \to \cX_{/\tau_{\le n} X}$.
	In conclusion, we have a commutative square
	\[ \begin{tikzcd}
	\cX_{/\tau_{\le n} X} \arrow{d}{\tau_{\le n}} & \cX_{/ X} \arrow{l}[swap]{j_!} \arrow{d}{\tau_{\le n}} \\
	\tau_{\le n} (\cX_{/\tau_{\le n} X}) & \tau_{\le n} (\cX_{/X}) \arrow{l}[swap]{\pi}
	\end{tikzcd} \]
	which allows to identify $\pi$ with $\tau_{\le n} \circ j_! \circ i$.
	Unraveling the definitions of the functors $\sigma$ and $\pi$, it follows that the natural transformation $\pi \circ \sigma \to \mathrm{Id}_{\tau_{\le n}(\cX_{/\tau_{\le n} X})}$  is an equivalence and therefore $\sigma$ is fully faithful.
	This proves (1).
	
	Let us now turn to (2).
	If $X$ is $m$-truncated and we are given an $n$-truncated map $f \colon Y \to X$, then we conclude that $Y$ is $\max\{n,m\}$-truncated.
	In particular, both $X$ and $Y$ are hypercomplete objects of $\cX$.
	We can therefore assume this last assertion to be satisfied in both situations.
	Consider the commutative diagram
	\[ \begin{tikzcd}
	Y \arrow[bend right = 20]{dr}[swap]{f} \arrow{r}{g} & \sigma(\tau_{\le n}(Y)) \arrow{r} \arrow{d}{f'} & \tau_{\le n} Y \arrow{d}{\tau_{\le n}(f)} \\
	{} & X \arrow{r} & \tau_{\le n} X
	\end{tikzcd} \]
	Observe that the map $f'$ is $n$-truncated and therefore that, in both the situations we are considering, $\sigma(\tau_{\le n}(Y))$ is an hypercomplete object of $\cX$.
	Thus, to conclude that $g$ is an equivalence, it will be sufficient to show that it is $\infty$-connected.
	\cref{lem:truncation_and_pullbacks} shows that $g$ induces an equivalence $\tau_{\le n} Y \simeq \tau_{\le n} \sigma(\tau_{\le n} Y)$, i.e.\ it induces isomorphisms on the $\pi_i$ for every $i \le n$.
	On the other side, both $f$ and $f'$ are $n$-truncated, and therefore $g$ induces isomorphisms on the $\pi_i$ for every $i \ge n+2$.
	
	We are left to deal with the case $i = n+1$.
	Consider the long exact sequence
	\[ \pi_{n+2}(\tau_{\le n} X) \to \pi_{n+1}(\sigma(\tau_{\le n}(Y)) ) \to \pi_{n+1}(X) \oplus \pi_{n+1}(\tau_{\le n}Y) \to \pi_{n+1}(\tau_{\le n} X) \]
	The vanishings of the involved homotopy groups show that $\pi_{n+1}(\sigma(\tau_{\le n}(Y))) \simeq \pi_{n+1}(X)$.
	Now set $F \coloneqq \fib(f)$ and $F' \coloneqq \fib(f')$. We have a morphism of fiber sequences
	\[ \begin{tikzcd}
	F \arrow{r} \arrow{d} & Y \arrow{d}{g} \arrow{r}{f} & X \arrow{d}{\mathrm{id}} \\
	F' \arrow{r} & \sigma(\tau_{\le n} Y) \arrow{r}{f'} & X
	\end{tikzcd} \]
	which in turn induces a morphism of long exact sequences
	\[ \begin{tikzcd}
	\pi_{n+1}(F) \arrow{d} \arrow{r} & \pi_{n+1}(Y) \arrow{d}{\pi_{n+1}(g)} \arrow{r}{\pi_{n+1}(f)} & \pi_{n+1}(X) \arrow{d}{\mathrm{id}} \\
	\pi_{n+1}(F') \arrow{r} & \pi_{n+1}(\sigma(\tau_{\le n} Y)) \arrow{r}{\pi_{n+1}(f')} & \pi_{n+1}(X)
	\end{tikzcd} \]
	Since both $f$ and $f'$ are $n$-truncated, we see that $\pi_{n+1}(F) = \pi_{n+1}(F') = 0$.
	On the other side, we already showed that $\pi_{n+1}(f')$ is an isomorphism.
	Therefore $\pi_{n+1}(g)$ is an isomorphism if and only if $\pi_{n+1}(f)$ is one, thus completing the proof of point (2).
	
	As for point (3), it is enough to observe that in the previous morphism of long exact above, the hypothesis forces $\pi_{n+1}(Y) = \pi_{n+1}(X) = 0$, so that \emph{every} $n$-truncated map $Y \to X$ satisfies the hypothesis of point (2).
	The lemma is thus proved.
\end{proof}

Let us now fix an $\infty$-topos $\cY$ and an object $X \in \cY$.
We will consider the $\infty$-topos $\cY_{/X}$.
The stabilization $\Sp(\cY_{/X})$ can be identified with the $\infty$-category of sheaves of spectra on $\cY_{/X}$, see \cite[Remark 1.2]{DAG-VII}.
Therefore, for every integer $n \in \mathbb Z$ we obtain functors
\[ \pi_n \colon \Sp(\cY_{/X}) \to \Ab(\tau_{\le 0}(\cY_{/X})) \]
On the other side, we also have a functor
\[ \Omega^\infty \colon \Sp(\cY_{/X}) \to \cY_{/X} \]
We will say that an object $M$ in $\Sp(\cY_{/X})$ is \emph{connective} if $\pi_n(M) = 0$ for $n > 0$.
On the other side, we will say that $M$ is \emph{coconnective} if $\Omega^\infty(M)$ is a discrete object of $\cY_{/X}$.
Equivalently, $M$ is coconnective if $\pi_n(M) = 0$ for every $n < 0$.
We warn the reader that, on the contrary of \cite{Lurie_Higher_algebra,DAG-VII}, we are using a \emph{cohomological notation} and not a homological one.

\begin{prop} \label{prop:t_structure_on_Fun}
	Let $\cY$ an $\infty$-topos and $X$ an object of $\cY$. Then:
	\begin{enumerate}
		\item The categories $(\Sp(\cY_{/X})_{\le 0}, \Sp(\cY_{/X})_{\ge 0})$ determine an accessible $t$-structure on $\Sp(\cY_{/X})$;
		\item the $t$-structure on $\Sp(\cY_{/X})$ is compatible with filtered colimits;
		\item the $t$-structure on $\Sp(\cY_{/X})$ is right complete.
	\end{enumerate}
	Suppose furthermore that $\cY$ is hypercomplete. Then:
	\begin{enumerate}
		\setcounter{enumi}{3}
		\item the functor $\pi_0$ determines an equivalence of the heart of this $t$-structure with the category of abelian group objects in $\tau_{\le 0}(\cY)_{/\pi_0\cO}$.
		\item The $t$-structure on $\Sp(\cY_{/X})$ is also left complete.
	\end{enumerate}
\end{prop}

\begin{proof}
	The first three statements are precisely a rewriting of \cite[Proposition 1.7]{DAG-VII}.
	The fifth one is the content of \cite[Warning 1.8]{DAG-VII}.
	We are left to deal with point (4).
	First of all using \cite[Proposition 1.7]{DAG-VII} we have a canonical identification
	\[	\Sp(\cY_{/X})^\heartsuit \simeq \Ab(\tau_{\le 0}( \cY_{/X}))	\]
	On the other side, combining points (1) and(3) of \cref{lem:truncation_and_overcategories} we obtain a fully faithful functor
	\[ \tau_{\le 0}(\cY)_{/\pi_0 X} \simeq \tau_{\le 0}(\cY_{/\pi_0 X} ) \to \tau_{\le 0}( \cY_{/ X}) \]
	This functor commutes with limits and therefore it induces a fully faithful functor
	\[ \Ab(\tau_{\le 0}(\cY)_{/\pi_0 X}) \to \Ab (\tau_{\le 0}( \cY_{/ X}) ) \]
	We claim that it is essentially surjective.
	Indeed, the forgetful functor
	\[ \Ab (\tau_{\le 0}( \cY_{/ X}) ) \to \tau_{\le 0} (\cY_{/X}) \]
	can be factored as
	\[ \Ab (\tau_{\le 0}(\cY_{/ X})) \to \tau_{\le 0}(\cY_{X // X}) \to \tau_{\le 0}(\cY_{/X}) \]
	Now, if $Y \to X$ is a $0$-truncated maps which admits a section, we see that it induces an isomorphism on each $\pi_1$.
	Therefore, point (2) of \cref{lem:truncation_and_overcategories} shows that $Y$ belongs to the essential image of $\tau_{\le 0}(\cY)_{/ \pi_0 \cO}$.
	The proof is now completed.
\end{proof}

We now want to endow $\cO\textrm{-} \Mod$ with a $t$-structure with similar properties.
\cite[Lemma 1.9]{Porta_GAGA_2015} shows that we have a fully faithful inclusion
\[ i \colon \Strloc_{\cT}(\cX)_{/\cO} \to \Fun(\cT, \cX)_{/\cO} \]
Since $i$ takes the final object to the final object, we deduce from \cite[Lemma 1.17]{Porta_GAGA_2015} that this functor commutes with limits and sifted colimits.
Therefore, composition with $i$ induces a fully faithful functor
\[ j \colon \cO \textrm{-} \Mod \to \Sp(\Fun(\cT, \cX)_{/\cO}) \]
The functor $j$ commutes with limits and filtered colimits.
Since the two categories $\cO\textrm{-} \Mod$ and $\Sp(\Fun(\cT, \cX)_{/\cO})$ are stable, we see that $j$ commutes also with finite colimits.
Therefore, $\cO \textrm{-} \Mod$ is a reflective \emph{and} coreflective subcategory of $\Sp(\Fun(\cT, \cX)_{/\cO})$.
It follows that it is stable under both limits and colimits computed in $\Sp(\Fun(\cT, \cX)_{/\cO})$.

As $\Fun(\cT, \cX)$ is an $\infty$-topos (which furthermore is hypercomplete if $\cX$ is), we can apply \cref{prop:t_structure_on_Fun} to get a $t$-structure on $\Sp(\Fun(\cT, \cX)_{/\cO})$. We are going to use this to induce a $t$-structure on $\cO\textrm{-} \Mod$. 

\begin{defin}
	Define $\cO \textrm{-} \Mod_{\ge 0}$ (resp.\ to $\cO \textrm{-} \Mod_{\le 0}$) as the full subcategory of $\cO \textrm{-} \Mod$ spanned by those elements $M$ such that $j(M)$ belongs to $\Sp(\Fun(\cT, \cX)_{/\cO})_{\ge 0}$ (resp.\ to $\Sp(\Fun(\cT, \cX)_{/\cO})_{\le 0}$).
	We will refer to the objects in $\cO \textrm{-} \Mod_{\ge 0}$ as \emph{coconnective $\cO$-modules} and to the objects in $\cO \textrm{-} \Mod_{\le 0}$ as \emph{connective $\cO$-modules}.
\end{defin}

\begin{prop} \label{prop:t_structure_Str}
	Let $\cT$ be a pregeometry, $\cX$ an $\infty$-topos and $\cO \in \Str_{\Tan}(\cX)$ a $\Tan$-structure on $\cX$. Then:
	\begin{enumerate}
		\item There exists a $t$-structure $(\cO \textrm{-} \Mod', \cO \textrm{-} \Mod_{\ge 0})$ on $\cO \textrm{-} \Mod$.
	\end{enumerate}
	Suppose furthermore that the pregeometry $\cT$ is compatible with $n$-truncations for every $n \ge 0$. Then:
	\begin{enumerate}
		\setcounter{enumi}{1}
		\item $\cO \textrm{-} \Mod' = \cO \textrm{-} \Mod_{\le 0}$;
		\item this $t$-structure is compatible with filtered colimits;
		\item this $t$-structure is right complete.
	\end{enumerate}
	Finally, suppose also that $\cX$ is hypercomplete. Then:
	\begin{enumerate}
		\setcounter{enumi}{4}
		\item the functor $\pi_0$ realizes an equivalence of the heart of this $t$-structure with abelian group objects in $\Strloc_{\cT}(\tau_{\le 0}\cX)_{/\pi_0 \cO}$.
		\item This $t$-structure is also left complete.
	\end{enumerate}
\end{prop}

\begin{proof}
	Since the functor $i \colon \Strloc_{\cT}(\cX)_{/\cO} \to \Fun(\cT, \cX)_{/\cO}$ preserves limits, we see that the square
	\[ \begin{tikzcd}
	\cO \textrm{-} \Mod \arrow{d}{\Omega^\infty} \arrow{r}{j} & \Sp(\Fun(\cT, \cX)_{/\cO}) \arrow{d}{\Omega^\infty} \\
	\Strloc_{\cT}(\cX)_{/\cO} \arrow{r}{i} & \Fun(\cT,\cX)_{/\cO}
	\end{tikzcd} \]
	commutes.
	Therefore, an element $M \in \cO \textrm{-} \Mod$ is coconnective if and only if $\Omega^\infty (j(\Omega(M))) \simeq i (\Omega^{\infty + 1}(M))$ is the final object of $\Fun(\cT, \cX)_{/\cO}$. This is equivalent to say that $\Omega^{\infty+1}(M)$ is the final object of $\Strloc_{\cT}(\cX)_{/\cO}$.
	Therefore the existence of the $t$-structure $(\cC, \cO \textrm{-} \Mod_{\ge 0})$ is a direct consequence of \cite[1.4.3.4]{Lurie_Higher_algebra}.
	
	We can now reason as in \cite[Proposition 1.7]{DAG-VII} to identify $\cC$.
	Namely, let $M \in \cO \textrm{-} \Mod$.
	We can identify $\cO \textrm{-} \Mod$ as the inverse limit of the diagram
	\[ \cdots \xrightarrow{\Omega} \Strloc_{\Tan}(\cX)_{\cO //\cO} \xrightarrow{\Omega} \Strloc_{\Tan}(\cX)_{\cO //\cO} \]
	Let $\cJ \coloneqq \mathbb N^{\mathrm{op}}$ be the poset indexing this diagram and let $p \colon \cC \to \cJ$ be the associated coCartesian fibration.
	We can identify $M$ with a Cartesian section of $p$, which amounts to a sequence of objects $M(n)$ together with equivalences $\gamma_n \colon M(n) \simeq \Omega M(n+1)$.
	Set $M'(n) \coloneqq \tau_{\le n - 1} M(n)$, where the truncation is taken in the presentable $\infty$-category $\Strloc_{\cT}(\cX)_{/\cO}$.
	Since $\cT$ is compatible with $n$-truncations for every $n \ge 0$, we can actually identify this truncation with the composition $\tau_{\le n-1}^{\cX} \circ M(n)$, which in turn coincides with the truncation in $\Fun(\cT, \cX)_{\cO // \cO}$ in virtue of \cref{lem:truncation_and_overcategories}.(2).
	It follows as in \cite[Proposition 1.7]{DAG-VII} that the equivalences $\gamma_n$ induce equivalences $\gamma'_n \colon M'(n) \simeq \Omega M(n+1)$, and therefore we can regard $M'(n)$ as an object in $\cO \textrm{-} \Mod$.
	Moreover, the canonical map $M \to M'$ exhibits $M'$ as a reflection of $M$ in $\cO \textrm{-} \Mod_{\ge 0}$. The proof given in loc.\ cit.\ (paired up with the description of truncations in $\Strloc_{\cT}(\cX)_{\cO // \cO}$ we gave above) shows that the map $j(M) \to j(M')$ exhibits $j(M')$ as a reflection of $j(M)$ in $\Sp(\Fun(\cT, \cX)_{\cO // \cO})_{\ge 0}$.
	Introduce the fiber sequence
	\[ M'' \to M \to M' \]
	in $\cO \textrm{-} \Mod$ (so that $M'' \in \cO \textrm{-} \Mod'$).
	Combining the previous observation with the fact that $j$ preserves fiber sequences, we conclude that $j(M'') \in \Sp(\Fun(\cT, \cX)_{\cO // \cO})_{\le -1}$, and therefore that $M'' \in \cO \textrm{-} \Mod_{\le -1}$.
	In conclusion, $M'' \in \cO \textrm{-} \Mod_{\le -1}$.
	This completes the proof of point (2)
	
	Now, points (3), (4) and (6) follow immediately from the fact that $j$ commutes with small colimits and \cref{prop:t_structure_on_Fun}.
	As for point (5), consider the following diagram
	\[ \begin{tikzcd}
	\Ab(\Strloc_{\cT}(\tau_{\le 0} \cX)_{/\tau_{\le 0} \cO}) \arrow[dashed]{d} \arrow{r} & \Ab(\Fun(\cT, \tau_{\le 0} \cX)_{/\tau_{\le 0} \cO}) \arrow{d} \\
	\Sp(\Strloc_{\cT}(\cX)_{/\cO})^\heartsuit \arrow{r} & \Sp(\Fun(\cT, \cX)_{/ \cO})^\heartsuit
	\end{tikzcd} \]
	The top horizontal morphism exists because the functor
	\[ \Strloc_{\cT}(\tau_{\le 0} \cX)_{/\tau_{\le 0} \cO} \to \Fun(\cT, \tau_{\le 0} \cX)_{/\tau_{\le 0} \cO} \]
	commutes with products. The right vertical functor is an equivalence in virtue of \cref{prop:t_structure_on_Fun}. We conclude that the dotted functor exists as well. Since the other three functors are fully faithful, the same goes for this functor. Finally, it is essentially surjective in an obvious way. The proof is therefore complete.
\end{proof}

Let now $\varphi \colon \cT' \to \cT$ be a morphism of pregeometries.
Let $\cX$ be an $\infty$-topos and let $\cO \in \Str_{\cT}(\cX)$.
\cite[Corollary 1.18]{Porta_GAGA_2015} shows that the induced functor
\[ \overline{\Phi} \colon \Strloc_{\cT}(\cX)_{/\cO} \to \Strloc_{\cT'}(\cX)_{/\cO \circ \varphi} \]
commutes with limits and sifted colimits.
Therefore, composition with $\overline{\Phi}$ induces a well defined functor
\[ \Phi_\cO \colon \Sp(\Strloc_{\cT}(\cX)_{/\cO}) \to \Sp(\Strloc_{\cT'}(\cX)_{/\cO \circ \varphi}) \]
The following result is a direct consequence of \cref{prop:t_structure_Str}:

\begin{cor} \label{cor:Phi_t_exact}
	Let $\cT$ and $\cT'$ be pregeometries which are compatible with $n$-truncations for every $n \ge 0$.
	Let $\varphi \colon \cT' \to \cT$ be a morphism between them.
	Let $\cX$ be an $\infty$-topos and let $\cO \in \Str_{\cT}(\cX)$.
	The induced functor
	\[ \Phi_\cO \colon \Sp(\Strloc_{\cT}(\cX)_{/\cO}) \to \Sp(\Strloc_{\cT'}(\cX)_{/\cO \circ \varphi}) \]
	is both left and right $t$-exact. In particular, its left adjoint $\Psi_\cO$ is right $t$-exact.
\end{cor}

\begin{proof}
	Consider the commutative diagram
	\[ \begin{tikzcd}
	\Sp(\Strloc_{\cT}(\cX)_{/\cO}) \arrow{r}{\Phi_\cO} \arrow{d}{j_\cT} & \Sp(\Strloc_{\cT'}(\cX)_{/\cO \circ \varphi}) \arrow{d}{j_{\cT'}} \\
	\Sp(\Fun(\cT, \cX)_{/\cO}) \arrow{r}{\widetilde{\Phi}_\cO} & \Sp(\Fun(\cT', \cX)_{/\cO \circ \varphi})
	\end{tikzcd} \]
	The vertical arrows are $t$-exact in virtue of \cref{prop:t_structure_Str} because both $\cT$ and $\cT'$ are compatible with $n$-truncations for $n \ge 0$.
	On the other side, we see that the functor $- \circ \varphi \colon \Fun(\cT, \cX)_{/\cO} \to \Fun(\cT, \cX)_{/\cO \circ \varphi}$ commutes with all limits, and therefore the induced functor $\widetilde{\Phi}$ is both left and right $t$-exact.
\end{proof}

\subsection{Reduction to the case of spaces} \label{subsec:stable_Morita}

We now turn to the first main reduction step needed in the proof of \cref{thm:equivalence_of_modules}.
Again, as the arguments given don't require any particular property of the pregeometry $\Tan$, we work with a general morphism of pregeometries $\varphi \colon \cT' \to \cT$, which will be fixed once and for all throughout the whole subsection.
We will furthermore fix an $\infty$-topos $\cX$ and $\cO \in \Str_{\cT}(\cX)$ a $\cT$-structure.
As we already discussed, composition with the natural forgetful functor
\[ \overline{\Phi}_{\cX} = \overline{\Phi} \colon \Strloc_{\cT}(\cX)_{/\cO} \to \Strloc_{\cT}(\cX)_{/\cO \circ \varphi} \]
commutes with finite limits and therefore induces a functor on the stabilizations:
\begin{equation} \label{eq:Phi_stable}
\Phi_{\cO, \cX} \colon \Sp(\Strloc_{\cT}(\cX)_{/\cO}) \to \Sp(\Strloc_{\cT}(\cX)_{/\cO \circ \varphi})
\end{equation}

\begin{rem}
	Recall from \cite{Lurie_Higher_algebra} that if $\cC$ is an $\infty$-category with terminal object $*$, then composition with the forgetful functor $\cC_* \coloneqq \cC_{* /} \to \cC$ induces an equivalence
	\[ \Sp(\cC_*) \to \Sp(\cC) \]
\end{rem}

In our setting, composition with $\varphi \colon \cT' \to \cT$ induces a well defined morphism
\[ \overline{\Phi}_{\cO, \cX} = \colon \Strloc_{\cT}(\cX)_{\cO //\cO} \to \Strloc_{\cT}(\cX)_{\cO \circ \varphi // \cO \circ \varphi} \]
fitting in a commutative diagram
\[ \begin{tikzcd}
\Strloc_{\cT}(\cX)_{\cO // \cO} \arrow{r}{\overline{\Phi}_{\cO, \cX}} \arrow{d} & \Strloc_{\cT'}(\cX)_{\cO \circ \varphi // \cO \circ \varphi} \arrow{d} \\
\Strloc_{\cT}(\cX)_{/ \cO} \arrow{r}{\overline{\Phi}_\cX} & \Strloc_{\cT'}(\cX)_{/\cO}
\end{tikzcd} \]
where the vertical morphisms are the forgetful functors.
The above remark shows therefore that passing to the stabilization $\overline{\Phi}_{\cO, \cX}$ and $\overline{\Phi}_\cX$ induce the same functor of stable $\infty$-categories.
Moreover, observe that the vertical morphisms reflect all limits and all connected colimits. Therefore, $\overline{\Phi}_{\cO, \cX}$ commutes with limits and sifted colimits.
Since undercategories of presentable categories are presentable (see \cite[Corollary 5.4.5.16]{HTT}), we conclude that $\overline{\Phi}_{\cO, \cX}$ admits a left adjoint, which we will denote by $\overline{\Psi}_{\cO, \cX}$.
We are now ready to state precisely the key result of this subsection:

\begin{thm} \label{thm:reduction_to_spaces}
	Let $\varphi \colon \cT' \to \cT$ be a morphism of pregeometries.
	Fix a geometric morphism of $\infty$-topoi $f\inv \colon \cX \rightleftarrows \cY \colon f_*$ let $\cO \in \Str_{\cT}(\cX)$.
	Set $\cO_1 \coloneqq \cO \circ \varphi$, $\cO_2 \coloneqq f\inv \circ \cO$ and $\cO_3 \coloneqq f\inv \circ \cO \circ \varphi$.
	Then the square
	\begin{equation} \label{eq:square_reduction_to_spaces}
	\begin{tikzcd}
	\Strloc_{\cT}(\cX)_{\cO //\cO} \arrow{r}{\overline{\Phi}_{\cO, \cX}} \arrow{d}[swap]{f\inv \circ -} & \Strloc_{\cT'}(\cX)_{\cO_1 //\cO_1} \arrow{d}{f\inv \circ -} \\
	\Strloc_{\cT}(\cY)_{\cO_2 // \cO_2} \arrow{r}{\overline{\Phi}_{\cO_2, \cY}} & \Strloc_{\cT'}(\cY)_{\cO_3 // \cO_3}
	\end{tikzcd}
	\end{equation}
	is left adjointable.
\end{thm}

The proof of this theorem is rather long and will occupy our attention for the most of this subsection.
Before starting to discuss it, let us describe the main consequences of this theorem.
Observe that in the situation of \cref{thm:reduction_to_spaces}, the functor
\[ f\inv \circ - \colon \Strloc_{\cT}(\cX)_{\cO // \cO} \to \Strloc_{\cT}(\cY)_{\cO_2 // \cO_2} \]
commutes with limits and filtered colimits (see \cite[Corollary 1.17]{Porta_GAGA_2015}).
Therefore it induces a functor on the stabilizations, which we will denote $\sigma_f$:
\[ \sigma_f \colon \Sp(\Strloc_{\cT}(\cX)_{\cO // \cO}) \to \Sp(\Strloc_{\cT}(\cY)_{\cO_2 // \cO_2}) \]
We have:

\begin{cor} \label{cor:reduction_to_spaces}
	Keeping the notations of \cref{thm:reduction_to_spaces}, the square
	\[ \begin{tikzcd}
	\Sp(\Strloc_{\cT}(\cX)_{\cO //\cO}) \arrow{r}{\Phi_{\cO, \cX}} \arrow{d}[swap]{\sigma_f} & \Sp(\Strloc_{\cT'}(\cX)_{\cO_1 //\cO_1}) \arrow{d}{\sigma_f} \\
	\Sp(\Strloc_{\cT}(\cY)_{\cO_2 // \cO_2}) \arrow{r}{\Phi_{\cO_2, \cY}} & \Sp(\Strloc_{\cT'}(\cY)_{\cO_3 // \cO_3})
	\end{tikzcd} \]
	is left adjointable.
\end{cor}

\begin{proof}
	This follows from \cref{thm:reduction_to_spaces} and \cref{prop:stabilization_of_left_adjointable_squares}.
\end{proof}

\begin{cor} \label{cor:first_reduction_step}
	Let $\varphi \colon \cT' \to \cT$ be a morphism of pregeometries.
	Let $\cX$ be an $\infty$-topos with enough points and let $\cO \in \Str_{\cT}(\cS)$ be a $\cT$-structure on $\cX$.
	Then the functor
	\[ \Psi_{\cO, \cX} \colon \Sp(\Strloc_{\cT'}(\cX)_{/\cO \circ \varphi}) \to \Sp(\Strloc_{\cT}(\cX)_{/\cO}) \]
	is fully faithful if and only if for every geometric point $x\inv \colon \cX \rightleftarrows \cS \colon x_*$ the functor
	\[ \Psi_{x\inv \cO, \cS} \colon \Sp(\Strloc_{\cT'}(\cS)_{/ x\inv \circ \cO \circ \varphi}) \to \Sp(\Strloc_{\cT}(\cS)_{/ x\inv \circ \cO}) \]
	is fully faithful.
\end{cor}

\begin{proof}
	The condition is necessary in virtue of \cref{cor:reduction_to_spaces} and \cref{prop:left_adjointable_preserves_unit}.
	On the other hand, suppose that the condition is satisfied.
	Let $\eta \colon \mathrm{Id} \to \Phi \circ \Psi$ be a unit morphism for the adjunction $(\Psi_{\cO, \cX}, \Phi_{\cO, \cX})$.
	Since $\cX$ has enough points, to show that $\eta$ is an equivalence is sufficient to show that for every geometric point  $x\inv \colon \cX \rightleftarrows \cS \colon x_*$ the morphism $x\inv(\eta)$ is an equivalence.
	Combining again \cref{cor:reduction_to_spaces} and \cref{prop:left_adjointable_preserves_unit}, we can identify $x\inv(\eta)$ with the unit of the adjunction $(\Psi_{x\inv \cO, \cS}, \Phi_{x\inv \cO, \cS})$, which is an equivalence by hypothesis.
\end{proof}

The rest of this subsection will be devoted to the proof of \cref{thm:reduction_to_spaces}.
Before explaining the general strategy, let us ease the notations.
Since $\cO$ will be fixed all throughout the section, we will suppress it in the notations $\overline{\Psi}_{\cO, \cX}$ and $\overline{\Phi}_{\cO, \cX}$.
However, since the $\overline{\Phi}_{\cX}$ already stands for the functor $\Strloc_{\cT}(\cX)_{/\cO} \to \Strloc_{\cT}(\cX)_{/\cO \circ \varphi}$, we will write $\overline{\Phi}'_{\cX}$ and $\overline{\Psi}'_{\cX}$ for the adjunction
\[ \Strloc_{\cT}(\cX)_{\cO // \cO} \rightleftarrows \Strloc_{\cT'}(\cX)_{\cO \circ \varphi // \cO \circ \varphi} \]

The general strategy consists of breaking up the square \eqref{eq:square_reduction_to_spaces} into simpler pieces and to prove that each of them is left adjointable.
We will then patch together the various informations using \cref{prop:horizontal_composition_left_adjointable}.

Let us start by explaining precisely how to break up the square \eqref{eq:square_reduction_to_spaces}.
Let $\varphi \colon \cT' \to \cT$ be the given morphism of pregeometries, let $\cX$ be an $\infty$-topos and let $\cO \in \Str_\cT(\cX)$.
At the beginning of this subsection we introduced the functor
\[ \overline{\Psi}'_\cX \colon \Strloc_{\cT'}(\cX)_{\cO \circ \varphi // \cO \circ \varphi} \to \Strloc_{\cT}(\cX)_{\cO // \cO}  \]
as left adjoint to the functor $\overline{\Phi}' \colon \Strloc_{\cT}(\cX)_{\cO // \cO} \to \Strloc_{\cT'}(\cX)_{\cO \circ \varphi // \cO \circ \varphi}$ induced by composition with $\varphi$.
We will decompose $\overline{\Psi}'$ as composition of simpler functors (see \cref{eq:decomposition_of_Psi}).
First of all, consider the inclusion functor
\[ i_\cX = i^{\cT}_\cX \colon \Strloc_{\cT}(\cX)_{\cO // \cO} \to \Fun(\cT, \cX)_{\cO // \cO} \]
which is fully faithful in virtue of \cite[Proposition 1.11]{Porta_GAGA_2015}.
\cite[Propositions 1.15 and 1.17]{Porta_GAGA_2015} show that this functor has a left adjoint, which we will denote by $L^{\cT}_{\cX}$ (or simply $L_\cX$ when the pregeometry is clear from the context).

Denote by
\[ \mathrm{Lan}_\varphi \colon \Fun(\cT', \cX) \to \Fun(\cT, \cX) \]
the left Kan extension along $\varphi$.
Set $\cO_1 \coloneqq \cO \circ \varphi$ and $\cO' \coloneqq \mathrm{Lan}_{\varphi}(\cO_1)$.
We have a canonical map
\[ \varepsilon_1 \colon \cO' = \mathrm{Lan}_{\varphi}(\cO \circ \varphi) \to \cO \]
induced by the counit of the adjunction $\mathrm{Lan}_\varphi \dashv - \circ \varphi$.
Left Kan extension along $\varphi$ induces a well-defined functor $\Fun(\cT', \cX)_{\cO_1 // \cO_1} \to \Fun(\cT, \cX)_{\cO' // \cO'}$ which we will continue to denote by $\mathrm{Lan}_\varphi$.
Introduce now the functor
\[ \Lambda_{\cX} \colon \Fun(\cT', \cX)_{\cO_1 // \cO_1} \to \Fun(\cT, \cX)_{\cO' // \cO} \]
formally defined as the composition
\[ \begin{tikzcd}
\Fun(\cT', \cX)_{\cO_1 // \cO_1} \arrow{r}{\mathrm{Lan}_\varphi} & \Fun(\cT, \cX)_{\cO' // \cO'} \arrow{r}{(\varepsilon_1)_!} & \Fun(\cT, \cX)_{\cO' // \cO}
\end{tikzcd} \]
where the second arrow is composition with $\varepsilon_1 \colon \cO' \to \cO$ (see the subsequent remark).

\begin{rem}
	Let $\cC$ be an $\infty$-category and let $f \colon X \to Y$ be a morphism in $\cC$. The canonical functor $\cC_{/X} \to \cC_{/Y}$ is formally defined as follows: let $p \colon \Delta^1 \to \cC$ be the functor classifying $f$. The evaluations at $0$ and at $1$ induce functors
	\[ \cC_{/X} \leftarrow \cC_{/p} \xrightarrow{\pi} \cC_{/Y} \]
	The dual of \cite[2.1.2.5]{HTT} shows that $\cC_{/p} \to \cC_{/X}$ is a trivial Kan fibration.
	We can therefore find a section $s \colon \cC_{/X} \to \cC_{/p}$ (uniquely determined up to homotopy) informally given by $(Z \to X) \mapsto (Z \to X \xrightarrow{f} Y)$.
	The composite $\pi \circ s$ is the functor we are looking for.
\end{rem}

\begin{lem} \label{lem:right_adjoint_to_Lambda}
	The functor $\Lambda_\cX$ admits a right adjoint
	\[ R_\cX \colon \Fun(\cT, \cX)_{\cO' // \cO} \to \Fun(\cT', \cX)_{\cO_1 // \cO_1} \]
	Moreover, the diagram
	\[ \begin{tikzcd}
	\Fun(\cT, \cX)_{\cO' // \cO} \arrow{r}{R_\cX} \arrow{d} & \Fun(\cT', \cX)_{\cO_1 // \cO_1} \arrow{d} \\
	\Fun(\cT, \cX) \arrow{r}{- \circ \varphi} & \Fun(\cT', \cX)
	\end{tikzcd} \]
	commutes.
\end{lem}

\begin{proof}
	Introduce the unit morphism $\eta_1 \colon \cO_1 \to \cO' \circ \varphi = \mathrm{Lan}_\varphi(\cO_1) \circ \varphi$.
	Consider the following commutative diagram
	\[ \begin{tikzcd}
	\Fun(\cT, \cX)_{\cO' // \cO} \arrow{r}{- \circ \varphi} \arrow{d} & \Fun(\cT', \cX)_{\cO' \circ \varphi // \cO_1} \arrow{r}{\eta_1^!} \arrow{d} & \Fun(\cT', \cX)_{\cO_1 // \cO_1} \arrow{d} \\
	\Fun(\cT, \cX) \arrow{r}{- \circ \varphi} & \Fun(\cT', \cX) \arrow{r}{\mathrm{id}} & \Fun(\cT', \cX) .
	\end{tikzcd} \]
	We define $R_\cX$ to be the composition of the upper horizontal morphisms.
	The universal property characterizing $\eta_1$ shows that, for every $B \in \Fun(\cT, \cX)_{\cO' // \cO}$ and every $A \in \Fun(\cT', \cX)_{\cO_1 // \cO_1}$ one has
	\begin{align*}
	\Map_{\cO_1 // \cO_1}(A, R_\cX(B)) & \simeq \Map_{\cO_1 // \cO \circ \varphi}(A, B \circ \varphi) \\
	& \simeq \Map_{\mathrm{Lan}_\varphi(\cO_1) // \cO}(\mathrm{Lan}_\varphi(A), B) \\
	& \simeq \Map_{\cO' // \cO}(L_\cX(A), B)
	\end{align*}
	and therefore we conclude that $L_\cX$ and $R_\cX$ are adjoint to each other.
\end{proof}

We leave the following lemma as an exercise to the reader:

\begin{lem} \label{lem:forgetful_and_pushout}
	Let $\cC$ be an $\infty$-category with pushouts.
	Let $\varepsilon \colon X \to Y$ be a morphism in $\cC$.
	The functor
	\[ \varepsilon^! \colon \cC_{Y /} \to \cC_{X /} \]
	given by composition with $\varepsilon$ admits a left adjoint $\pi \colon \cC_{X /} \to \cC_{Y /}$ given by pushout along $\varepsilon$.
	Dually, if $\cC$ has pullbacks, the functor $\varepsilon_! \colon \cC_{/ X} \to \cC_{/ Y}$ has a right adjoint given by pullback along $\varepsilon$.
\end{lem}

\begin{rem} \label{rem:R_X_alternative_description}
	\cref{lem:forgetful_and_pushout} provides another useful characterization of $R_\cX$.
	Indeed, $\Lambda_\cX = {(\varepsilon_1)_{!}} \circ \mathrm{Lan}_\varphi $, so that $R_\cX \simeq G_\varphi \circ (\varepsilon_1)\inv$, where $G_\varphi$ denotes a right adjoint to $\mathrm{Lan}_\varphi \colon \Fun(\cT', \cX)_{\cO_1 // \cO_1} \to \Fun(\cT, \cX)_{\cO' // \cO'}$.
	Unraveling the definitions, we therefore see that for every $A \in \Fun(\cT, \cX)_{\cO' // \cO}$, $R_\cX(A)$ fits into the pullback square
	\[ \begin{tikzcd}
	R_\cX(A) \arrow{r} \arrow{d} & A \circ \varphi \arrow{d} \\
	\cO_1 \arrow{r}{\eta_1} & \cO' \circ \varphi
	\end{tikzcd} \]
	where $\eta_1 \colon \cO_1 \to \cO' \circ \varphi = \mathrm{Lan}_\varphi(\cO_1) \circ \varphi$ is the unit of the adjunction.
\end{rem}

The previous lemma provides us with a left adjoint
\[ \pi \colon \Fun(\cT, \cX)_{\cO' // \cO} \to \Fun(\cT, \cX)_{\cO // \cO} \]
We can further compose with the reflection
\[ L^{\cT}_\cX \colon \Fun(\cT, \cX)_{\cO // \cO} \to \Strloc_{\cT}(\cT, \cX)_{\cO // \cO} \]
In this way, we obtained a functor
\[ L^{\cT}_\cX \circ \pi \circ \Lambda_\cX \colon \Fun(\cT', \cX)_{\cO_1 // \cO_1} \to \Strloc_{\cT}(\cT, \cX)_{\cO // \cO} \]
which is a left adjoint, being the composition of three left adjoints.
\cref{lem:right_adjoint_to_Lambda} implies that its right adjoint can be simply identified with precomposition with $\varphi$. It is therefore clear that this right adjoint factors through the inclusion
\[ i^{\cT}_{\cX} \colon \Strloc_{\cT'}(\cX)_{\cO_1 // \cO_1} \to \Fun(\cT, \cX)_{\cO_1 // \cO_1} . \]
Thus, we obtain the desired decomposition:
\begin{equation} \label{eq:decomposition_of_Psi}
\overline{\Psi}'_\cX \simeq L^{\cT}_{\cX} \circ \pi \circ \Lambda_\cX \circ i^{\cT'}_\cX .
\end{equation}
We now turn to the behavior of each of these components with respect to a geometric morphism of $\infty$-topoi $f\inv \colon \cX \rightleftarrows \cY \colon f_*$.

\subsubsection{The functors $i^\cT_\cX$ and $L^{\cT}_\cX$}

We will begin our analysis with the functor $i^{\cT}_\cX$ and its adjoint, $L^{\cT}_\cX$.

\begin{lem} \label{lem:lex_composition_preserves_structures}
	Let $\cT$ be a pregeometry and let $g \colon \cX \to \cY$ be a left exact functor between $\infty$-topoi.
	Let $\cO \in \Str_{\cT}(\cX)$.
	Then the square:
	\[ \begin{tikzcd}
	\Strloc_{\cT}(\cX)_{\cO // \cO} \arrow{r}{i^{\cT}_{\cX}} \arrow{d}{g \circ -} & \Fun(\cT, \cX)_{\cO // \cO} \arrow{d}{g\circ -} \\
	\Strloc_{\cT}(\cY)_{g \circ \cO // g \circ \cO} \arrow{r}{i^{\cT}_{\cY}} & \Fun(\cT, \cY)_{g \circ \cO // g \circ \cO}
	\end{tikzcd} \]
	commutes.
\end{lem}

\begin{proof}
	In virtue of \cite[Proposition 1.11]{Porta_GAGA_2015}, it is sufficient to remark that composition with $g$ sends functors which preserve products to functors that preserve products, and that moreover the image under $g \circ -$ of a local morphism is still a local morphism.
	This follows at once from left exactness of $g$.
\end{proof}

We fix now a geometric morphism of $\infty$-topoi $f\inv \colon \cX \rightleftarrows \cY \colon f_*$ and a $\cT$-structure $\cO \in \Str_{\cT}(\cX)$.
The above lemma provides us with a commutative diagram
\[ \begin{tikzcd}
\Strloc_{\cT}(\cX)_{\cO // \cO} \arrow{r}{i^{\cT}_{\cX}} \arrow{d}{f\inv \circ -} & \Fun(\cT, \cX)_{\cO // \cO} \arrow{d}{f\inv \circ -} \\
\Strloc_{\cT}(\cY)_{f\inv \circ \cO // f\inv \circ \cO} \arrow{r}{i^{\cT}_{\cY}} & \Fun(\cT, \cY)_{f\inv \circ \cO // f\inv \circ \cO}
\end{tikzcd} \]
Set temporarily $\overline{\cO}  \coloneqq f_* \circ f\inv \circ \cO$.
Note that $\overline{\cO}$ is not a $\cT$-structure on $\cX$.
Nevertheless, we have a canonical map $\alpha \colon \cO \to \overline{\cO}$ induced by the unit of the adjunction $f\inv \dashv f_*$.
Pulling back along this morphism produces a functor
\[ \begin{tikzcd}
\rho_* \colon \Fun(\cT, \cY)_{f\inv \circ \cO // f\inv \circ \cO} \arrow{r}{f_* \circ -} & \Fun(\cT, \cX)_{\overline{\cO} // \overline{\cO}} \arrow{r}{\alpha^*} & \Fun(\cT, \cX)_{\cO // \cO}
\end{tikzcd} \]
Inspection shows that this is a right adjoint to $f\inv \circ -$.
Since $f_*$ commutes with limits, we see that $f_* \circ -$ takes product preserving functors to product preserving functors, and local morphisms to local morphisms.
Since pullbacks commute with limits, we see that $\alpha^*$ shares the same properties.
Therefore, the same argument given in the proof of \cref{lem:lex_composition_preserves_structures} shows that this composition induces a factorization
\[ \begin{tikzcd}
\Strloc_{\cT}(\cY)_{f\inv \circ \cO // f\inv \circ \cO} \arrow[dotted]{d}{\rho_*} \arrow{r}{i^{\cT}_{\cY}} & \Fun(\cT, \cX)_{f\inv \circ \cO // f\inv \circ \cO} \arrow{d}{\rho_*} \\
\Strloc_{\cT}(\cX)_{\cO // \cO} \arrow{r}{i^{\cT}_{\cX}} & \Fun(\cT, \cX)_{\cO // \cO}
\end{tikzcd} \]
We summarize this discussion in the following lemma:

\begin{lem} \label{lem:geometric_adjunction_on_str}
	Let $\cT$ be a pregeometry, $f\inv \colon \cX \rightleftarrows \cY \colon f_*$ be a geometric morphism of $\infty$-topoi and $\cO \in \Str_{\cT}(\cX)$.
	Then the adjunction
	\[ f\inv \circ - \colon \Fun(\cT, \cX)_{\cO // \cO} \rightleftarrows \Fun(\cT, \cY)_{f\inv \circ \cO // f\inv \circ \cO} \colon \rho_* \]
	induces an adjunction
	\[ f\inv \circ - \colon \Strloc_{\cT}(\cX)_{\cO // \cO} \rightleftarrows \Strloc_{\cT}(\cY)_{f\inv \circ \cO // f\inv \circ \cO} \colon \rho_* \]
	and both the diagrams
	\[ \begin{tikzcd}
	\Strloc_{\cT}(\cX)_{\cO // \cO} \arrow{r}{i^\cT_{\cX}} \arrow[xshift = -0.60ex ]{d}[swap]{f\inv \circ -} & \Fun(\cT, \cX)_{\cO // \cO} \arrow[xshift = -0.60ex]{d}[swap]{f\inv \circ -} \\
	\Strloc_{\cT}(\cY)_{f\inv \circ \cO // f\inv \circ \cO} \arrow[xshift = 0.60ex]{u}[swap]{\rho_*} \arrow{r}{i^{\cT}_{\cY}} & \Fun(\cT, \cX)_{f\inv \circ \cO // f\inv \circ \cO} \arrow[xshift = 0.60ex]{u}[swap]{\rho_*}
	\end{tikzcd} \]
	are commutative.
\end{lem}

\begin{proof}
	We already proved the existence of the functors
	\begin{gather*}
	f\inv \circ - \colon \Strloc_{\cT}(\cX)_{\cO // \cO} \to \Strloc_{\cT}(\cY)_{f\inv \circ \cO // f\inv \circ \cO}, \\
	\rho_* \colon \Strloc_{\cT}(\cY)_{f\inv \circ \cO // f\inv \circ \cO} \to \Strloc_{\cT}(\cX)_{\cO // \cO} .
	\end{gather*}
	The fact that they are adjoint to each other follows at once from the fully faithfulness of $i^{\cT}_{\cX}$ and of $i^{\cT}_{\cY}$.
\end{proof}

\begin{prop} \label{prop:left_adjointable_strloc_fun}
	Let $\cT$ be a pregeometry, $f\inv \colon \cX \leftrightarrows \cY \colon f_*$ a geometric morphism of $\infty$-topoi and $\cO \in \Str_\cT(\cX)$ a $\cT$-structure.
	Set $\cO_2 \coloneqq f\inv \circ \cO$.
	The commutative square
	\[ \begin{tikzcd}
	\Strloc_{\cT}(\cX)_{\cO // \cO} \arrow{r}{i_\cX} \arrow{d}{f\inv \circ -} & \Fun(\cT, \cX)_{\cO // \cO} \arrow{d}{f\inv \circ -} \\
	\Strloc_{\cT}(\cY)_{\cO_2 // \cO_2} \arrow{r}{i_\cY} & \Fun(\cT, \cY)_{\cO_2 // \cO_2}
	\end{tikzcd} \]
	is left adjointable.
\end{prop}

\begin{proof}
	We remark that in order to prove that the square
	\[ \begin{tikzcd}
	\Strloc_{\cT}(\cX)_{\cO // \cO} \arrow{d}[swap]{f\inv \circ -} & \Fun(\cT, \cX)_{\cO // \cO} \arrow{l}[swap]{L_\cX} \arrow{d}{f\inv \circ -} \\
	\Strloc_{\cT}(\cX)_{\cO_2 // \cO_2} & \Fun(\cT, \cY)_{\cO_2 // \cO_2} \arrow{l}[swap]{L_\cY}
	\end{tikzcd} \]
	commutes, it suffices to prove that the diagram of right adjoints
	\[ \begin{tikzcd}
	\Strloc_{\cT}(\cX)_{\cO // \cO} \arrow{r}{i^{\cT}_\cX} & \Fun(\cT, \cX)_{\cO // \cO} \\
	\Strloc_{\cT}(\cX)_{\cO_2 // \cO_2} \arrow{u}{\rho_*} \arrow{r}{i^{\cT}_\cY} & \Fun(\cT, \cY)_{\cO_2 // \cO_2} \arrow{u}[swap]{\rho_*}
	\end{tikzcd} \]
	is commutative. This follows from \cref{lem:geometric_adjunction_on_str}.
	To complete the proof, fix $\cF \in \Fun(\cT, \cX)_{\cO // \cO}$ and let $\eta \colon \cF \to \widetilde{\cF}$ be a reflection of $\cF$ in $\Strloc_{\cT}(\cX)_{\cO // \cO}$.
	We need to prove that
	\[ f\inv(\eta) \colon f\inv(\cF) \to f\inv(\widetilde{\cF}) \]
	exhibits $f\inv(\widetilde{\cF})$ as a reflection of $f\inv(\cF)$ in $\Strloc_{\cT}(\cY)_{f\inv \cO // f\inv \cO}$.
	For every $\cG \in \Strloc_{\cT}(\cY)_{f\inv \cO // f\inv \cO}$, we have
	\begin{align*}
	\Map^{\Fun}_{f\inv \circ \cO // f\inv \circ \cO}(f\inv \circ \cF, \cG) & \simeq \Map^{\Fun}_{\cO // \cO}(\cF, \rho_*(\cG)) \\
	& \simeq \Map^{\Str}_{\cO // \cO}(\widetilde{\cF}, \rho_*(\cG)) \\
	& \simeq \Map^{\Str}_{f\inv \circ \cO // f\inv \circ \cO}(f\inv \circ \widetilde{\cF}, \cG) .
	\end{align*}
	where we denoted with $\Map^{\Fun}_{f\inv \circ \cO // f\inv \circ \cO}$ (resp.\ $\Map^{\Fun}_{\cO // \cO}$, $\Map^{\Str}_{\cO // \cO}$, $\Map^{\Str}_{f\inv \circ \cO // f\inv \circ \cO}$) the mapping space in $\Fun(\cT, \cY)_{f\inv \circ \cO // f\inv \circ \cO}$ (resp.\ in $\Fun(\cT, \cX)_{\cO // \cO}$, $\Strloc_\cT(\cX)_{\cO // \cO}$, $\Strloc_{\cT}(\cY)_{f\inv \circ \cO // f\inv \circ \cO}$).
	The proposition is now completely proved.
\end{proof}

\subsubsection{The functor $\Lambda_\cX$}

Keeping the notations we previously introduced, we now turn to the functor
\[ \Lambda_{\cX} \colon \Fun(\cT', \cX)_{\cO_1 // \cO_1} \to \Fun(\cT, \cX)_{\cO' // \cO} \]
Recall that it has a right adjoint $R_\cX$ informally given by precomposition with the morphism of pregeometries $\varphi \colon \cT' \to \cT$ (see \cref{lem:right_adjoint_to_Lambda}).

\begin{lem} \label{lem:left_Kan_left_adjointable}
	The diagram
	\[ \begin{tikzcd}
	\Fun(\cT, \cX) \arrow{r}{- \circ \varphi} \arrow{d}[swap]{f\inv \circ -} & \Fun(\cT', \cX) \arrow{d}{f\inv \circ -} \\
	\Fun(\cT, \cY) \arrow{r}{- \circ \varphi} & \Fun(\cT', \cY)
	\end{tikzcd} \]
	is left adjointable.
\end{lem}

\begin{proof}
	First of all observe that the diagram
	\[ \begin{tikzcd}
	\Fun(\cT, \cX) \arrow{r}{- \circ \varphi} & \Fun(\cT', \cX) \\
	\Fun(\cT, \cY) \arrow{u}{f_* \circ -} \arrow{r}{- \circ \varphi} & \Fun(\cT', \cY) \arrow{u}[swap]{f_* \circ -}
	\end{tikzcd} \]
	obviously commutes.
	It follows that the diagram of left adjoints
	\[ \begin{tikzcd}
	\Fun(\cT, \cX) \arrow{d}[swap]{f\inv \circ -} & \Fun(\cT', \cX) \arrow{d}{f\inv \circ -} \arrow{l}[swap]{\mathrm{Lan}_\varphi} \\
	\Fun(\cT, \cY) & \Fun(\cT', \cY) \arrow{l}[swap]{\mathrm{Lan}_\varphi}
	\end{tikzcd} \]
	Let $F \in \Fun(\cT, \cX)$ and choose a counit transformation $\varepsilon_F \colon \mathrm{Lan}_\varphi(F \circ \varphi) \to F$.
	Using \cref{prop:Beck_Chevalley_fibration_formulation}, it will be enough to prove that
	\[ f\inv(\varepsilon_F) \colon f\inv \circ \mathrm{Lan}_\varphi(F \circ \varphi) \simeq \mathrm{Lan}_\varphi(f\inv \circ F \circ \varphi)  \to f\inv \circ F \]
	is equivalent to $\varepsilon_{f\inv \circ F}$.
	Fix $X \in \cT$ and let $(\varphi \downarrow X)$ be the comma category introduced in \cref{def:comma_category}.
	Combining \cite[4.3.3.2, 4.3.2.15]{HTT} we see that the evaluation of the counit morphism $\mathrm{Lan}_\varphi(F \circ \varphi) \to F$ on $X$ can be explicitly constructed as the canonical map\footnote{The comma category $(\varphi \downarrow X)$ is implicitly defined in \cite[§4.3.3]{HTT}. It amounts to consider $\mathrm{Cyl}(\varphi) \coloneqq (\cT' \times \Delta^1) \coprod_{\cT' \times \{1\}} \cT$ and then set $(\varphi \downarrow X) \coloneqq (\cT' \times \{0\})_{/X} \subset \mathrm{Cyl}(\varphi)_{/X}$.}
	\[ \mathrm{Lan}_\varphi(F \circ \varphi)(X) = \colim_{Y \in (\varphi \downarrow X)} F(\varphi(Y)) \to F(X) . \]
	At this point, we see that the fact that $f\inv$ commutes with colimits implies that it preserves the counit transformation.
\end{proof}

Introduce now $\cO_2 \coloneqq f\inv \circ \cO$ and $\cO_3 \coloneqq \cO_2 \circ \varphi = f\inv \circ \cO_1$.
Furthermore we let $\cO_2' \coloneqq \mathrm{Lan}_\varphi(\cO_3)$.
\cref{lem:left_Kan_left_adjointable} shows that $\cO_2' \simeq f\inv \circ \cO'$, so that composition with $f\inv$ induces a well defined functor
\[ f\inv \circ - \colon \Fun(\cT, \cX)_{\cO' // \cO} \to \Fun(\cT, \cX)_{\cO_2' // \cO_2} \]

\begin{prop} \label{prop:Lambda_left_adjointable}
	The square
	\begin{equation} \label{eq:Lambda_left_adjointable}
	\begin{tikzcd}
	\Fun(\cT, \cX)_{\cO' // \cO} \arrow{r}{R_\cX} \arrow{d}{f\inv \circ -} & \Fun(\cT', \cX)_{\cO_1 // \cO_1} \arrow{d}{f\inv \circ -} \\
	\Fun(\cT, \cY)_{\cO'_2 // \cO_2} \arrow{r}{R_\cY} & \Fun(\cT', \cY)_{\cO_3 // \cO_3}
	\end{tikzcd}
	\end{equation}
	commutes and is left adjointable.
\end{prop}

\begin{proof}
	Introduce the unit morphisms $\eta_1 \colon \cO_1 \to \cO' \circ \varphi = \mathrm{Lan}_\varphi(\cO_1) \circ \varphi$ and $\eta_2 \colon \cO_3 \to \cO_2' \circ \varphi = \mathrm{Lan}_\varphi(\cO_3) \circ \varphi$.
	\cref{lem:left_Kan_left_adjointable} shows that $f\inv \circ \eta_1 = \eta_2$.
	The proof of \cref{lem:right_adjoint_to_Lambda} shows that we can factor the diagram \eqref{eq:Lambda_left_adjointable} as
	\[ \begin{tikzcd}
	\Fun(\cT, \cX)_{\cO' // \cO} \arrow{d}{f\inv \circ -} \arrow{r}{-\circ \varphi} & \Fun(\cT', \cX)_{\cO' \circ \varphi // \cO} \arrow{d}{f\inv \circ -} \arrow{r}{\eta_1^!} & \Fun(\cT', \cX)_{\cO_1 // \cO_1} \arrow{d}{f\inv \circ -} \\
	\Fun(\cT, \cY)_{\cO_2' // \cO_2} \arrow{r}{-\circ \varphi} & \Fun(\cT', \cY)_{\cO_2' \circ \varphi // \cO_2} \arrow{r}{\eta_2^!} & \Fun(\cT', \cY)_{\cO_3 // \cO_3} ,
	\end{tikzcd} \]
	where $\eta_1^!$ and $\eta_2^!$ denote the composition with $\eta_1$ and $\eta_2$ respectively.
	Now, the left square obviously commutes, and the right one commutes in virtue of the above observation.
	The first statement is therefore proved.
	
	We now introduce the diagram
	\[ \begin{tikzcd}
	\Fun(\cT', \cX)_{\cO_1 // \cO_1} \arrow{r}{\mathrm{Lan}_\varphi} \arrow{d}{f\inv \circ -} & \Fun(\cT, \cX)_{\cO' // \cO'} \arrow{r}{(\varepsilon_1)_!} \arrow{d}{f\inv \circ -} & \Fun(\cT, \cX)_{\cO' // \cO} \arrow{d}{f\inv \circ -} \\
	\Fun(\cT', \cY)_{\cO_3 // \cO_3} \arrow{r}{\mathrm{Lan}_\varphi} & \Fun(\cT, \cY)_{\cO_2' // \cO_2'} \arrow{r}{(\varepsilon_2)_!} & \Fun(\cT, \cY)_{\cO_2' // \cO_2}
	\end{tikzcd} \]
	where $\varepsilon_2 \colon \cO_2' = \mathrm{Lan}_\varphi(\cO_2 \circ \varphi) \to \cO_2$ denotes the counit of the adjunction $\mathrm{Lan}_\varphi \dashv - \circ \varphi$.
	By definition, the horizontal composites equal $\Lambda_\cX$ and $\Lambda_\cY$ respectively.
	\cref{lem:left_Kan_left_adjointable} shows that both these squares are commutative.
	
	Let now $\gamma_1 \colon \mathrm{Id} \to R_\cX \circ \Lambda_\cX$ and $\gamma_2 \colon \mathrm{Id} \to R_\cY \circ \Lambda_\cY$ be the units of the respective adjunctions. We will complete the proof by showing that $f\inv \circ \gamma_1$ is equivalent to $\gamma_2$.
	Fix $F \in \Fun(\cT', \cX)_{\cO_1 // \cO_1}$.
	Since the right adjoint of $(\varepsilon_1)_!$ is the pullback along $\varepsilon_1 \colon \cO' \to \cO$ and since $- \circ \varphi$ commutes with pullbacks, we can identify $\gamma_1(F)$ as the morphism fitting in the diagram
	\[ \begin{tikzcd}
	F \arrow[bend left = 15]{drr}{\eta_F} \arrow[bend right = 20]{ddr} \arrow{dr}{\gamma_1} \\
	{} & R_\cX(\Lambda_\cX(F)) \arrow{r} \arrow{d} & \mathrm{Lan}_\varphi(F) \circ \varphi \arrow{d} \\
	{} & \cO_1 \arrow{r}{\eta_1} & \cO' \circ \varphi
	\end{tikzcd} \]
	where the square is cartesian (see \cref{rem:R_X_alternative_description}).
	Since $f\inv$ is left exact and $f\inv \circ -$ preserves both $\eta_1$ and $\eta_F$ by \cref{lem:left_Kan_left_adjointable}, we conclude that $f\inv \circ \gamma_1$ is equivalent to $\gamma_2$, completing the proof.
\end{proof}

\subsubsection{The proof of \cref{thm:reduction_to_spaces}.}

At last, all the preliminaries have been dealt with, and we are ready to deal with the proof of our main result.

\begin{proof}[Proof of \cref{thm:reduction_to_spaces}.]
	We proceed in two steps.
	
	\emph{Step 1.} We  prove that the outer rectangle in the following commutative diagram
	\[ \begin{tikzcd}
	\Strloc_{\cT}(\cX)_{\cO // \cO} \arrow{r}{i^{\cT}_\cX} \arrow{d}{f\inv \circ -} & \Fun(\cT, \cX)_{\cO // \cO} \arrow{r}{(\varepsilon_1)^!} \arrow{d}{f\inv \circ -} & \Fun(\cT, \cX)_{\cO' // \cO} \arrow{r}{R_\cX} \arrow{d}{f\inv \circ -} & \Fun(\cT', \cX)_{\cO_1 // \cO_1} \arrow{d}{f\inv \circ -} \\
	\Strloc_{\cT}(\cY)_{\cO_2 // \cO_2} \arrow{r}{i^{\cT}_\cY} & \Fun(\cT, \cY)_{\cO_2 // \cO_2} \arrow{r}{(\varepsilon_2)^!} & \Fun(\cT, \cY)_{\cO_2' // \cO_2} \arrow{r}{R_\cY} & \Fun(\cT', \cY)_{\cO_3 // \cO_3}
	\end{tikzcd} \]
	is left adjointable.
	In virtue of \cref{prop:horizontal_composition_left_adjointable}, it will be sufficient to show that the three commutative squares are left adjointable.
	We dealt with the one on the left in \cref{prop:left_adjointable_strloc_fun} and with the one on the right in \cref{prop:Lambda_left_adjointable}.
	As for the middle one, it is left adjointable because $f\inv$ commutes with pushouts and in virtue of \cref{lem:left_Kan_left_adjointable}.
	
	\emph{Step 2.} Consider now the diagram
	\[ \begin{tikzcd}
	\Strloc_{\cT}(\cX)_{\cO // \cO} \arrow{r}{\overline{\Phi}'_\cX} \arrow{d}{f\inv \circ -} & \Strloc_{\cT'}(\cX)_{\cO_1 // \cO_1} \arrow{r}{i^{\cT'}_\cX} \arrow{d}{f\inv \circ -} & \Fun(\cT', \cX)_{\cO_1 // \cO_1} \arrow{d}{f\inv \circ -} \\
	\Strloc_{\cT}(\cY)_{\cO_2 // \cO_2} \arrow{r}{\overline{\Phi}'_\cY} & \Strloc_{\cT'}(\cY)_{\cO_3 // \cO_3} \arrow{r}{i^{\cT'}_\cY} & \Fun(\cT', \cY)_{\cO_3 // \cO_3} .
	\end{tikzcd} \]
	We proved in Step 1 that the outer rectangle is left adjointable, and we can use \cref{prop:left_adjointable_strloc_fun} to see that the square on the right is left adjointable as well.
	We want to deduce that the same goes for the left square.
	
	In \eqref{eq:decomposition_of_Psi} we showed that $\overline{\Psi}'_\cX = L^{\cT}_{\cX} \circ \pi \circ \Lambda_\cX \circ i^{\cT'}_{\cX}$.
	To ease notations, introduce $F_\cX \coloneqq \pi \circ \Lambda_\cX \circ i^{\cT'}_\cX$ and $F_\cY \coloneqq \pi \circ \Lambda_\cY \circ i^{\cT'}_\cY$.
	Denote by $G_\cX$ and $G_\cY$ the respective right adjoints.
	We can use Step 1.\ to deduce that the diagram
	\[ \begin{tikzcd}
	\Strloc_{\cT}(\cX)_{\cO // \cO} \arrow{d}{f\inv \circ -} & \Strloc_{\cT'}(\cX)_{\cO_1 // \cO_1} \arrow{l}[swap]{\overline{\Psi}'_\cX} \arrow{d}{f\inv \circ -} \\
	\Strloc_{\cT}(\cY)_{\cO_2 // \cO_2} & \Strloc_{\cT'}(\cY)_{\cO_3 // \cO_3} \arrow{l}[swap]{\overline{\Psi}'_\cY}
	\end{tikzcd} \]
	commutes.
	Let now $\delta_1 \colon \overline{\Psi}'_\cX \circ \overline{\Phi}'_\cX \to \mathrm{Id}$ and $\delta_2 \colon \overline{\Psi}'_\cY \circ \overline{\Phi}'_\cY \to \mathrm{Id}$ be the counits of the respective adjunctions.
	To complete the proof it will be sufficient to show that $f\inv \circ \delta_1$ is equivalent to $\delta_2$.
	Fix $A \in \Strloc_{\cT}(\cX)_{\cO // \cO}$.
	Then
	\[ i^{\cT}_\cX(\delta_1(A)) \colon i^{\cT}_\cX( \overline{\Psi}'_\cX( \overline{\Phi}'_\cX(A))) \to i^{\cT}_\cX(A) \]
	is uniquely determined by the property of making the diagram
	\[ \begin{tikzcd}
	F_\cX(G_\cX( i^{\cT}_{\cX} (A))) \arrow{r}{\alpha_1} \arrow{dr}[swap]{\beta_1} & L^{\cT}_{\cX}(F_\cX(G_\cX(i^{\cT}_{\cX} (A)))) \arrow{d}{i^{\cT}_{\cX}(\delta_1(A))} \simeq i^{\cT'}_\cX(\overline{\Psi}'(\overline{\Phi}'_\cX(A))) \\
	& i^{\cT}_\cX(A)
	\end{tikzcd} \]
	where $\beta_1$ is the counit of the adjunction $F_\cX \dashv G_\cX$ and $\alpha_1$ is the unit of the adjunction $L^{\cT}_\cX \dashv i^{\cT}_\cX$.
	The morphism $\delta_2(f\inv \circ A)$ has a similar characterization.
	We conclude that $f\inv \circ \delta_1$ is equivalent to $\delta_2$ by using \cref{lem:left_Kan_left_adjointable} and the Step 1.
\end{proof}

\subsection{A further reduction step}

From now on, we will specialize to the case $\cT = \Tan$.
Let $(\cX, \cO_\cX)$ be a derived \canal space.
In this case, we already know that the functor
\[ \Phi_\cO \colon \Sp(\Strloc_{\Tan}(\cX)_{/\cO_\cX}) \to \Sp(\Strloc_{\cTdisc}(\cX)_{/\cO_\cX\alg}) \]
is conservative. Therefore, \cref{thm:equivalence_of_modules} is equivalent to the fully faithfulness of the left adjoint $\Psi_\cO$.
We have the following refinement of \cite[Lemma 3.2]{Porta_GAGA_2015}:

\begin{lem} \label{lem:derived_canal_space_enough_points}
	Let $(\cX, \cO_\cX)$ be a derived \canal space.
	Then $\cX$ has enough points.
\end{lem}

\begin{proof}
	Choose elements $U_i \in \cX$ such that the total morphism $\coprod U_i \to \mathbf 1_\cX$ is an effective epimorphism and each $\cX_{/U_i}$ is equivalent to the $\infty$-topos of sheaves on some \canal space.
	Then each $\cX_{/U_i}$ has enough points and every geometric point $x\inv \colon \cX_{/U_i} \rightleftarrows \cS \colon x_*$ gives rise via composition with $u_i\inv \colon \cX \rightleftarrows \cX_{/U_i} \colon (u_i)_*$ to a geometric point of $\cX$.
	Let $\alpha \colon F \to G$ be a morphism in $\cX$. Since $\coprod U_i \to \mathbf 1_\cX$ is an effective epimorphism, the general descent theory for $\infty$-topoi (see \cite[6.1.3.9]{HTT}) shows that $\alpha$ is an equivalence if and only if each $\alpha_i \coloneqq u_i\inv(\alpha)$ is an equivalence.
	Therefore $\alpha$ is an equivalence if and only if $x\inv (u_i\inv(\alpha))$ is an equivalence for every index $i$ and every geometric point $x\inv \colon \cX_{/U_i} \rightleftarrows \cS \colon x_*$.
	In conclusion, $\cX$ has enough points.
\end{proof}

The previous lemma shows that the hypotheses of \cref{cor:first_reduction_step} are satisfied in the situation of \cref{thm:equivalence_of_modules}. Thus, it is enough to consider the case $(\cS, \cO)$, where $\cO = x\inv \cO_\cX$ for some geometric point $x\inv \colon \cX \rightleftarrows \cS \colon x_*$.
Since from this point on there won't be any ambiguity concerning the $\infty$-topos $\cX$, we will replace the notation $\overline{\Psi}'_\cX$ with the more appropriate $\overline{\Psi}_\cO$.

Let $M \in \Sp(\Strloc_{\cTdisc}(\cS)_{/\cO\alg}) = \cO\alg \textrm{-} \Mod$ be a spectrum object.
Using \cite[1.4.2.24]{Lurie_Higher_algebra} we can represent $\Sp(\Strloc_{\cTdisc}(\cS)_{\cO\alg // \cO\alg})$ as the inverse limit of the diagram
\[ \ldots \xrightarrow{\Omega} \Strloc_{\cTdisc}(\cS)_{\cO\alg // \cO\alg} \xrightarrow{\Omega} \Strloc_{\cTdisc}(\cS)_{\cO\alg // \cO\alg} \]
Let $\mathcal J \coloneqq \mathbb N^{\mathrm{op}}$ be the poset parametrizing this diagram and let $p \colon \cC \to \mathcal J$ be the coCartesian fibration classifying it.
Similarly, let $q \colon \cD \to \mathcal J$ be the coCartesian fibration classifying the diagram
\[ \ldots \xrightarrow{\Omega} \Strloc_{\Tan}(\cS)_{\cO // \cO} \xrightarrow{\Omega} \Strloc_{\Tan}(\cS)_{\cO // \cO} \]
The functor $\overline{\Phi}_\cO$ induces a natural transformation of such diagrams and therefore it determines a functor $G_\cO \colon \cD \to \cC$ with the property of taking $q$-coCartesian morphisms to $p$-coCartesian ones.
This functor $G_\cO$ admits a left adjoint $F_\cO$ which no longer takes $p$-coCartesian edges to $q$-coCartesian ones.
Informally, $F_\cO$ sends an object $(A, n) \in \cC$ to $(\overline{\Psi}_\cO(A), n) \in \cD$.

In virtue of \cite[3.3.3.2]{HTT} we can identify $M$ with a coCartesian section of $p$, that is with a sequence of objects $\{M(n)\}$ of $\Strloc_{\cTzar}(\cS)_{\cO\alg // \cO\alg}$ and equivalences $\gamma_n \colon M(n) \simeq \Omega M(n+1)$.
The previous consideration shows that the composition $F_\cO \circ M$ is no longer a coCartesian section of $q$, at least a priori.
In the language of strongly excisive functors we are saying that $\overline{\Psi}_{\cO} \circ M$ is not a spectrum in an obvious way.
Nevertheless, it is true, as we are going to show.
We begin with a handy lemma.

\begin{lem}
	Let $A \in \mathrm{CAlg}_{\mathbb C}$ be a local connective $\mathbb E_\infty$-ring.
	The morphism of pregeometries induce an equivalence $\Sp(\Strloc_{\cTzar}(\cS)_{/A}) \simeq \Sp(\Strloc_{\cTdisc}(\cS)_{/A})$.
\end{lem}

\begin{proof}
	Combining \cite[Remark 4.2.12]{DAG-V} and \cite[Lemma 1.11]{Porta_GAGA_2015}, we deduce that the canonical functor $\Strloc_{\cTzar}(\cS)_{/A} \to \Strloc_{\cTdisc}(\cS)_{/A}$
	is fully faithful.
	It follows that the induced functor
	\[ \Sp(\Strloc_{\cTzar}(\cS)_{/A}) \to \Sp(\Strloc_{\cTdisc}(\cS)_{/A}) \]
	is fully faithful as well.
	It will be enough to show that it is essentially surjective.
	In other words, we have to prove that for every $M \in \Sp(\Strloc_{\cTdisc}(\cS)_{/A})$, $\Omega^\infty(M)$ is a local ring.
	It follows from \cite[Remark 4.2.13]{DAG-V} that this is a property that depends only on $\pi_0(\Omega^\infty(M))$, and \cite[7.3.4.17]{Lurie_Higher_algebra} shows that we can identify this with the split square-zero extension of $\pi_0(A)$ by $\pi_0(M)$.
	As $\pi_0(A)$ is local, the same goes for this split square-zero extension, thus completing the proof.
\end{proof}

\begin{prop} \label{prop:Psi_preserves_spectra}
	Let $M \colon \cS^{\mathrm{fin}}_* \to \Strloc_{\cTzar}(\cS)_{\cO\alg // \cO\alg}$ be a spectrum object.
	Then the composite $\overline{\Psi}_\cO \circ M$ defines again a spectrum object of $\Strloc_{\Tan}(\cS)_{\cO // \cO}$.
\end{prop}

\begin{proof}
	It will be more convenient to reason in terms of the coCartesian fibrations $p \colon \cC \to \mathcal J$ and $q \colon \cD \to \mathcal J$ introduced above.
	We will therefore identify $M$ with a coCartesian section of $p$, and we will prove that $F_\cO \circ M$ is a coCartesian section of $q$.
	For every $n$, we apply \cref{prop:relative_flatness} to the object $M(n) \simeq \Omega^{\infty - n}(M) \in \Strloc_{\cTzar}(\cS)_{\cO\alg // \cO\alg}$ in order to deduce that the canonical map
	\[ M(n) \to \overline{\Phi}_\cO (\overline{\Psi}_\cO (M(n))) \]
	is flat in the derived sense.
	Therefore, in order to prove that it is an equivalence, it is sufficient to show that its $\pi_0$ is an equivalence.
	This will be proved in \cref{prop:equivalence_discrete_modules}.
	
	Assuming the result for the moment, we use the fact that $\overline{\Phi}_\cO$ commutes with limits to form the following commutative diagram in $\Strloc_{\cTzar}(\cS)_{\cO\alg // \cO\alg}$:
	\[ \begin{tikzcd}
	M(n) \arrow{d} \arrow{r} & \Omega M(n+1) \arrow{d} \\
	\overline{\Phi}_\cO (\overline{\Psi}_\cO (M(n))) \arrow{r} & \Omega \overline{\Phi}_\cO (\overline{\Psi}_\cO (M(n+1))) .
	\end{tikzcd} \]
	We already argued that the vertical morphisms are equivalences. Since $M$ was a coCartesian section of $p$ the top horizontal morphism is an equivalence as well.
	It follows that the bottom horizontal morphism is an equivalence, and since $\overline{\Phi}$ is conservative (see \cite[Proposition 1.9]{DAG-IX}), we conclude that the canonical map
	\[ \overline{\Psi}_\cO (M(n)) \to \overline{\Psi}_\cO (M(n+1)) \]
	is an equivalence too. Therefore $F_\cO \circ M$ is a coCartesian section of $q$.	
\end{proof}

\begin{cor} \label{cor:Psi_preserves_spectra}
	Let $(\cX, \cO)$ be a derived \canal space.
	Then the diagram
	\[ \begin{tikzcd}
	\Sp(\Strloc_{\cTdisc}(\cX)_{/\cO\alg}) \arrow{r}{\Psi_{\cO}} \arrow{d}{\Omega^\infty_{\mathrm{alg}}} & \Sp(\Strloc_{\Tan}(\cX)_{\cO}) \arrow{d}{\Omega^\infty_{\mathrm{an}}} \\
	\Strloc_{\cTdisc}(\cX)_{\cO\alg // \cO\alg} \arrow{r}{\overline{\Psi}_{\cO}} & \Strloc_{\Tan}(\cX)_{\cO // \cO}
	\end{tikzcd} \]
	commutes.
\end{cor}

\begin{proof}
	The functor $\Psi_\cO$ is, by definition, the first Goodwillie derivative of $\overline{\Psi}_\cO$.
	We can therefore identify $\Psi_\cO$ with the composite
	\[ \begin{tikzcd}
	\mathrm{Exc}_*(\cS^{\mathrm{fin}}_*, \Strloc_{\cTdisc}(\cX)_{/\cO\alg}) \arrow[dashed]{r}{\Psi_\cO} \arrow{d} & \mathrm{Exc}_*(\cS^{\mathrm{fin}}_*, \Strloc_{\Tan}(\cX)_{/\cO}) \\
	\Fun(\cS^{\mathrm{fin}}_*, \Strloc_{\cTdisc}(\cX)_{/\cO\alg}) \arrow{r}{\overline{\Psi}_\cO \circ -} & \Fun(\cS^{\mathrm{fin}}_*, \Strloc_{\Tan}(\cX)_{/\cO\alg}) \arrow{u}[swap]{P_1}
	\end{tikzcd} \]
	where the left vertical morphism is the natural inclusion and $P_1$ denotes its left adjoint (see \cite[6.1.1.10]{Lurie_Higher_algebra}).
	Since $\cX$ has enough points, we can apply \cref{prop:Psi_preserves_spectra} to deduce that if $M \in \mathrm{Exc}_*(\cS^{\mathrm{fin}}_*, \Strloc_{\cTdisc}(\cX)_{/\cO\alg})$, then $\overline{\Psi}_\cO \circ M$ is a strongly excisive functor on the nose.
	Therefore \cite[6.1.1.35]{Lurie_Higher_algebra} shows that $\Psi_\cO(M) \simeq P_1( \overline{\Psi}_\cO \circ M ) \simeq \overline{\Psi}_\cO \circ M$.
	The conclusion follows.
\end{proof}

\begin{rem}
	Actually, \cref{prop:Psi_preserves_spectra} implies that the square
	\[ \begin{tikzcd}
	\Sp(\Strloc_{\cTdisc}(\cX)_{/\cO\alg}) \arrow{d}{\Omega^\infty_{\mathrm{alg}}} & \Sp(\Strloc_{\Tan}(\cX)_{\cO}) \arrow{l}[swap]{\Phi_{\cO}} \arrow{d}{\Omega^\infty_{\mathrm{an}}} \\
	\Strloc_{\cTdisc}(\cX)_{\cO\alg // \cO\alg} & \Strloc_{\Tan}(\cX)_{\cO // \cO} \arrow{l}[swap]{\overline{\Phi}_{\cO}}
	\end{tikzcd} \]
	is left adjointable.
\end{rem}

We are finally in position to achieve the proof of \cref{thm:equivalence_of_modules}.
Let $\eta \colon \mathrm{Id} \to \Psi_\cO \circ \Phi_\cO$ be the unit of the adjunction $\Psi_\cO \dashv \Phi_\cO$.
We already argued that it is enough to show that $\eta$ is an equivalence.

Now let $M \in \Strloc_{\cTdisc}(\cS)_{\cO\alg // \cO\alg} \simeq \cO\alg \textrm{-} \Mod$ be the spectrum object representing $\cO\alg$ seen as module over itself.
Since $M$ is a compact generator of $\cO\alg \textrm{-} \Mod$ and since both $\Psi_\cO$ and $\Phi_\cO$ commute with arbitrary colimits, it will be enough to prove that $\eta_M \colon M \to \Phi_\cO(\Psi_\cO(M))$ is an equivalence.
Observe that $M$ is connective; \cref{cor:Phi_t_exact} implies that $\Phi_\cO(\Psi_\cO(M))$ is connective as well.
As consequence, we deduce that $\eta_M$ is an equivalence if and only if $\Omega^\infty_{\mathrm{alg}}(\eta_M)$ is an equivalence.
Invoking \cref{cor:Psi_preserves_spectra}, we are reduced to prove that the map
\[ \Omega^\infty_{\mathrm{alg}}(M) \to \Omega^\infty_{\mathrm{alg}}( \overline{\Phi}_\cO (\overline{\Psi}_\cO (M))) \simeq \overline{\Phi}_\cO (\Omega^{\infty}_{\mathrm{an}}(\overline{\Psi}_\cO (M))) \simeq \overline{\Psi}_\cO (\overline{\Psi}_\cO (\Omega_{\mathrm{alg}}^\infty(M))) \]
is an equivalence.
Using \cref{prop:relative_flatness}, we see that this map is flat in the derived sense.
Hence, it will be enough to show that it is an equivalence on $\pi_0$, and this is once more a consequence of \cref{prop:equivalence_discrete_modules}.
Modulo this fact, the proof of \cref{thm:equivalence_of_modules} is now achieved.

\subsection{The case of discrete local analytic rings}

Let us maintain the notations introduced in the previous subsection.
We will complete the proof of \cref{thm:equivalence_of_modules} by proving that for every $\cO \in \Str_{\Tan}(\cS)$ the canonical map
\[ \pi_0(M(n)) \to \pi_0 (\overline{\Phi}_{\cO}(\overline{\Psi}_{\cO}(M(n)) )) \]
is an isomorphism.
Since $\Tan$ is compatible with $0$-truncations, we see that $\pi_0 (\overline{\Phi}_{\cO}(\overline{\Psi}_{\cO}(M(n)) )) \simeq  \overline{\Phi}_{\cO}(\pi_0 \overline{\Psi}_{\cO}(M(n)) )$.

Using \cref{lem:truncation_and_overcategories}.(2) we can identify $\Strloc_{\Tan}(\rSet)_{\pi_0 \cO // \pi_0 \cO}$ (resp.\ $\Strloc_{\cTdisc}(\rSet)_{\pi_0 \cO\alg // \pi_0 \cO\alg}$) with the full subcategory of $\Strloc_{\Tan}(\cS)_{\cO // \cO}$ (resp.\ $\Strloc_{\cTdisc}(\cS)_{\cO\alg // \cO\alg}$) spanned by discrete objects.
Let us denote by
\[ \overline{\Phi}^0_\cO \colon \Strloc_{\Tan}(\rSet)_{\pi_0 \cO // \pi_0 \cO} \to \Strloc_{\cTdisc}(\rSet)_{\pi_0 \cO\alg // \pi_0 \cO\alg} \]
the functor induced by precomposition with $\varphi$.
Since both $\cTdisc$ and $\Tan$ are compatible with $0$-truncations, we see that the diagram
\[ \begin{tikzcd}
\Strloc_{\cTdisc}(\cS)_{\cO\alg // \cO\alg} & \Strloc_{\Tan}(\cS)_{\cO // \cO} \arrow{l}[swap]{\overline{\Phi}_\cO} \\
\Strloc_{\cTdisc}(\rSet)_{\pi_0 \cO\alg // \pi_0 \cO\alg} \arrow{u}[swap]{i} & \Strloc_{\Tan}(\rSet)_{\pi_0 \cO // \pi_0 \cO} \arrow{u}[swap]{j} \arrow{l}[swap]{\overline{\Phi}^0_\cO}
\end{tikzcd} \]
commute, where the vertical morphism are the fully faithful inclusions we discussed above.
Let us denote by $\overline{\Psi}^0_\cO$ the left adjoint of $\overline{\Phi}^0_\cO$, which exists for the same formal reasons guaranteeing the existence of $\overline{\Psi}_\cO$.
Then, passing to left adjoints in the above diagram we conclude that the following diagram
\[ \begin{tikzcd}
\Strloc_{\cTdisc}(\cS)_{\cO\alg // \cO\alg} \arrow{d}{\pi_0} \arrow{r}{\overline{\Psi}_\cO} & \Strloc_{\Tan}(\cS)_{\cO // \cO} \arrow{d}{\pi_0} \\
\Strloc_{\cTdisc}(\rSet)_{\pi_0 \cO\alg // \pi_0 \cO\alg} \arrow{r}{\overline{\Psi}^0_\cO} & \Strloc_{\Tan}(\rSet)_{\pi_0 \cO // \pi_0 \cO}
\end{tikzcd} \]
commutes as well.
Let us finally remark that a left adjoint to $\overline{\Phi}^0$ can be explicitly described by $\overline{\Psi}^0 \coloneqq \pi_0 \circ \overline{\Psi}_\cO \circ i$.
Indeed, if $R \in \Strloc_{\cTdisc}(\rSet)_{\pi_0 \cO\alg // \pi_0 \cO\alg}$ and $A \in \Strloc_{\Tan}(\rSet)_{\pi_0 \cO // \pi_0 \cO}$, then
\begin{align*}
\Map(\pi_0( \overline{\Psi}_\cO(i(R)) ), A) & \simeq \Map(\overline{\Psi}_\cO(i(R)), j(A)) \\
& \simeq \Map(i(R), \overline{\Phi}_\cO(j(A))) \\
& \simeq \Map(i(R), i(\overline{\Phi}^0_\cO(A) ) \simeq \Map(R, \overline{\Phi}^0_\cO(A))
\end{align*}

Thus, coming back to our initial setting, we conclude from this discussion that
\[ \pi_0 \overline{\Psi}_\cO(M(n)) \simeq \pi_0( \overline{\Psi}_\cO(\pi_0(M(n)))) . \]
Observe that for $n > 0$ we have
\[ \pi_0(M(n)) \simeq \cO\alg . \]
Thus, we only need to take care of the case $n = 0$.
Now, $\pi_0(M(0))$ can be identified with the split square-zero extension $\pi_0 \cO\alg \oplus \pi_0 \cO\alg$ (see \cite[7.3.4.17]{Lurie_Higher_algebra}).
It will therefore be sufficient to prove the following result:

\begin{prop} \label{prop:equivalence_discrete_modules}
	Let $A \in \Strloc_{\Tan}(\rSet)$.
	Let $M \in A\alg \textrm{-} \Mod^\heartsuit$ and let $A_M \coloneqq A\alg \oplus M$ be the split square-zero extension of $A$ by $M$.
	The canonical map $A_M \to \overline{\Phi}^0_A ( \overline{\Psi}^0_A(A_M) )$ is an equivalence.
\end{prop}

The proof of this proposition passes through the following rather explicit construction, that will be handy also in the future:

\begin{construction} \label{jacobian_construction}
	Let $A \in \Strloc_{\Tan}(\rSet)$ and let $M \in A\alg \textrm{-} \Mod^\heartsuit$.
	We define a functor $\Jac_A(M) \colon \Tan \to \rSet$ as follows.
	If $U \subset \mathbb C^n$ is an open subset, we define
	\[ \Jac_A(M)(U) \coloneqq A(U) \times M^n  \]
	If $f \colon U \to V$ is a holomorphic map, with $U \subset \mathbb C^n$ and $V \subset \mathbb C^m$, we define
	\[ \Jac_A(M)(f) \colon \Jac_A(M)(U) \to \Jac_A(M)(V) \]
	in the following way. First, we associate to $f$ its jacobian function
	\[ \Jac(f) \colon U \to \mathbb \mathrm{M}_{n,m}(\mathbb C) \simeq \mathbb C^{nm} , \]
	where $\mathrm M_{n,m}(\mathbb C)$ denotes the $n$-by-$m$ matrices with values in $\mathbb C$.
	Applying $A$ we obtain a morphism
	\[ A(\Jac(f)) \colon A(U) \to A(\mathbb C^{nm}) \simeq \mathrm{M}_{n,m}(A\alg) \]
	Now, if $t \in M^n$ is a column vector, we obtain a map
	\[ A(U) \times M^n \to M^m \]
	defined by
	\[ (a,t) \mapsto A(\Jac(f))(a) \cdot t \]
	We then set
	\[ \Jac_A(M)(f)(a,t) \coloneqq (A(f)(a), A(\Jac(f))(a) \cdot t) \]
	The chain rule for the jacobian of a holomorphic function shows that in this way we indeed obtain a functor
	\[ \Jac_A(M) \colon \Tan \to \rSet . \]
	We will refer to $\Jac_A(M)$ as the \emph{jacobian construction of $M$ (relative to $A$)}.
\end{construction}

The next lemma collects the most basic facts about $\Jac_A(M)$:

\begin{lem} \label{lem:jacobian_construction}
	Let $A \in \Str_{\Tan}(\rSet)$.
	\begin{enumerate}
		\item The assignment $M \mapsto \Jac_A(M)$ defines a functor
		\[ \Jac_A \colon A \textrm{-} \Mod^\heartsuit \to \Fun(\Tan, \rSet). \]
		\item For every $M \in A \textrm{-} \Mod^\heartsuit$ there are (local) natural transformation $\pi \colon \Jac_A(M) \to A$ and $s \colon A \to \Jac_A(M)$.
		\item For every $M \in A \textrm{-} \Mod^\heartsuit$, $\Jac_A(M)$ is a $\Tan$-structure in $\rSet$.
		\item For every $M \in A \textrm{-} \Mod^\heartsuit$, $\Jac_A(M)\alg$ is canonically identified with the split square-zero extension $A\alg \oplus M$, and under this identification $\pi\alg$ becomes the projection $A\alg \oplus M \to A\alg$, while $s\alg$ becomes the null derivation.
	\end{enumerate}
\end{lem}

\begin{proof}
	We leave the first statement as an exercise to the reader.
	We define the natural transformations $\pi$ and $s$ as follows. For every open subset $U \subset \mathbb C^n$, we set
	\[ \pi_U(a,m) \coloneqq a \in A(U), \qquad s_U(a) \coloneqq (a,0) \in \Jac_A(M)(U) = A(U) \times M^n \]
	It is straightforward to check that both $\pi$ and $s$ are natural transformations.
	If $V \subset \mathbb C^n$ is an open subset and $j \colon U \subset V \subset \mathbb C^n$ is an open immersion, the jacobian function $\mathrm{Jac}(j)$ is the constant function associated to the identity matrix in $\mathrm M_n(\mathbb C)$.
	Therefore the function
	\[ \Jac_A(M)(j) \colon \Jac_A(M)(U) \to \Jac_A(M)(V) \]
	is the monomorphism $(a,m) \mapsto (a,m)$. This shows that the square
	\[ \begin{tikzcd}
	\Jac_A(M)(U) \arrow{r} \arrow{d} & \Jac_A(M)(V) \arrow{d} \\
	A(U) \arrow{r} & A(V)
	\end{tikzcd} \]
	is a pullback square.
	Therefore, $\pi$ is a local transformation.
	Since $\Jac_A(M)$ commutes with products by definition, it follows that $\Jac_A(M)$ is a $\Tan$-structure (the reader can consult the proof of \cite[Proposition 1.11]{Porta_GAGA_2015}).
	It follows that $s$ is a local transformation as well.
	
	We are left to prove point (4). The only thing we need to prove is that the multiplicative structure on $\Jac_A(M)\alg$ coincides with the one on the split square-zero extension on $A\alg \oplus M$.
	Consider the multiplication
	\[ m \colon \mathbb C \times \mathbb C \to \mathbb C , \]
	$m(z,w) = zw$. Its jacobian is the function $\Jac(m) \colon \mathbb C \times \mathbb C \to \mathrm M_{2, 1}(\mathbb C)$ defined by
	\[ (z,w) \mapsto (w,z) \]
	Therefore, the multiplication on $\Jac_A(M)(\mathbb C)$ is given by the rule
	\[ ((a_1, t_1), (a_2, t_2)) \mapsto (a_1 a_2, a_2 t_1 + a_1 t_2) \]
	which we recognize being the multiplication rule on the split square-zero extension $A\alg \oplus M$.
	The rest of the proof is straightforward.
\end{proof}

\cref{prop:equivalence_discrete_modules} will be a direct consequence of the next proposition:

\begin{prop}
	Let $A \in \Str_{\Tan}(\rSet)$ and let $M \in A\alg \textrm{-} \Mod^\heartsuit$.
	Then there exists a canonical identification $\Jac_A(M) \simeq \overline{\Psi}^0_A(A\alg \oplus M)$.
\end{prop}

\begin{proof}
	It will be enough to show that for every $B \in \Strloc_{\Tan}(\cS)_{A // A}$ composition with the identification $A\alg \simeq \Jac_A(M)\alg$ provided by \cref{lem:jacobian_construction} induces an isomorphism
	\[ \Hom_{A // A} (\Jac_A(M), B) \to \Hom_{A\alg // A\alg}(A\alg \oplus M, B\alg) . \]
	Since the forgetful functor $\overline{\Phi}'_0$ is conservative and commutes with limits, it is faithful.
	Therefore the above morphism is a monomorphism.
	We are left to show that it is surjective. Let $\alpha \colon A\alg \oplus M \to B\alg$ be a morphism in $\CRing_{A\alg // A\alg}$.
	We define a lifting $\widetilde{\alpha} \colon \Jac_A(M) \to B$ as follows.
	Denote by $q \colon B  \to A$ the structural morphism and set
	\[ N \coloneqq \ker(q\alg \colon B\alg \to A\alg) \in A\alg \textrm{-} \Mod^\heartsuit \]
	The morphism $\alpha$ induces a map of $A\alg$-modules
	\[ \beta \colon M \to N \]
	Since $q$ is a local transformation, we see that for every open $U \subset \mathbb C^n$ one has
	\[ B(U) \simeq A(U) \times N^n . \]
	Using this remark, we define $\widetilde{\alpha}_U \colon \Jac_A(M)(U) \to B(U)$ as
	\[ \widetilde{\alpha}_U(a,t) \coloneqq (a, \beta(t)) \in A(U) \times N^n \]
	We will now prove that $\widetilde{\alpha}$ is a natural transformation. Since $\widetilde{\alpha}\alg = \alpha$ by construction, this will complete the proof.
	
	Let $U \subset \mathbb C^n$ and $V \subset \mathbb C^m$. We will say that a holomorphic map $f \colon U \to V$ is \emph{good} if the diagram
	\[ \begin{tikzcd}
	\Jac_A(M)(U) \arrow{r}{\widetilde{\alpha}_U} \arrow{d}[swap]{\Jac_A(M)(f)} & B(U) \arrow{d}{B(f)} \\
	\Jac_A(M)(V) \arrow{r}{\widetilde{\alpha}_V} & B(V)
	\end{tikzcd} \]
	is commutative.
	A morphism $f \colon U \to V \times W$ is good if and only if its components $\mathrm{pr}_1 \circ f$ and $\mathrm{pr}_2$ are good.
	Every open immersion $j \colon V \to W$ is good; additionally, if $f \colon U \to V$ is a morphism, then $f$ is good if and only if $j \circ f$ is good, as it follows from the fact that $B(j) \colon B(V) \to B(W)$ is a monomorphism.
	Combining these two remarks, we are immediately reduced to show that every holomorphic function  $f \colon U \to \mathbb C$ is good.
	Pick an element $(a,r) \in B(U)$, with $a = (a_1, \ldots, a_n) \in A(U)$ and $r = (r_1, \ldots, r_n) \in N^n$. If $\sigma_B$ denotes the functional spectrum for $\Jac_A(M)$ introduced in \cite[§2.1]{Porta_GAGA_2015} and $\sigma_A$ denotes the one of $A$, we have \cite[Proposition 2.14]{Porta_GAGA_2015}:
	\[ \sigma_B(a,r) = \sigma_A(a) \in U \]
	Set $p \coloneqq \sigma(a)$ and choose an open neighborhood $U_1$ of $p$ in $U$ and an open neighborhood $U_2$ of $0$ in $\mathbb C^n$ in such a way that the addition
	\[ + \colon U_1 \times U_2 \to \mathbb C^n \]
	factors through $U$. We can assume without loss of generality that both $U_1$ and $U_2$ are open polydisks.
	Consider the function $\widetilde{f} \colon U_1 \times U_2 \to \mathbb C$ defined by
	\[ \widetilde{f}(z,w) \coloneqq f(z + w) \]
	Since $f$ is a holomorphic function, we can expand $\widetilde{f}$ in power series as
	\[ \widetilde{f}(z,w) = f(z) + \sum_{i = 1}^n \frac{\partial f}{\partial z_i}(z) w_i + \sum_{i,j = 1}^n g_{ij}(z,w) w_i w_j \]
	Writing $(a,r) = (a,0) + (0,r)$ and observing that $(a,0) \in B(U_1)$ and $(0,m) \in B(U_2)$, we obtain
	\begin{align*}
	B(f)(a,n) & = B(f)(a,0) + \sum_{i = 1}^n B \left( \frac{\partial f}{\partial z_i} \right)(a,0) \cdot r_i + \sum_{i,j = 1}^n B(g_{ij})(a,r) \cdot r_i \cdot r_j \\
	& = (A(f)(a), 0) + \sum_{i = 1}^n A \left( \frac{\partial f}{\partial z_i} \right) (a) \cdot r_i + \sum_{i,j = 1}^n B(g_{ij}(a,r)) \cdot r_i \cdot r_j
	\end{align*}
	We can now complete the proof as follows: let $(a,t) \in \Jac_A(M)(U)$. Then $\widetilde{\alpha}_U(a,t) = (a, \beta(t))$ and since $\alpha$ was a ring homomorphism to begin with, we see that
	\[ \beta(t_i) \cdot \beta(t_j) = \beta(t_i \cdot t_j) = 0 \]
	Therefore the above formula yields
	\[ B(f)(a, \beta(t)) = (A(f)(a), 0) + \sum_{i = 1}^n A \left( \frac{\partial f}{\partial z_i} \right) (a) \cdot \beta(t_i) \]
	which coincides with $\widetilde{\alpha}_{\mathbb C}(\Jac_A(M)(f)(a,t))$.
	Thus, the morphism $f \colon U \to \mathbb C$ is good, and the proof is complete.
\end{proof}

\begin{rem}
	The previous proposition has the following non-trivial consequence. Let $\CRing_{A\alg // A\alg}^{2 \textrm{-} \mathrm{nil}}$ be the full subcategory of $\CRing_{A\alg // A\alg}$ spanned by split square-zero extensions of $A\alg$.
	Then the functor $\overline{\Psi}'_0$, which doesn't commute with limits in general, commutes with products of objects in $\CRing_{A\alg // A\alg}^{2 \textrm{-} \mathrm{nil}}$.
	In turn, this implies that the restriction of $\overline{\Psi}'_0$ to $\CRing_{A\alg // A\alg}^{2 \textrm{-} \mathrm{nil}}$ induces by composition a functor
	\[ \Ab(\CRing_{A\alg // A\alg}^{2 \textrm{-} \mathrm{nil}}) \to \Ab(\Strloc_{\Tan}(\rSet)_{A // A}) \]
	In a sense, this is the key idea underlying the proof of \cref{thm:equivalence_of_modules}.
	The path we took to actually prove this theorem stemmed out of this basic observation.
	The closest higher categorical statement we managed to prove is \cref{prop:Psi_preserves_spectra}.
	Finally, we remark that it seems unlikely that the $1$-categorical proof we gave here translate verbatim in the setting of \cref{thm:equivalence_of_modules}.
	The reason is that the jacobian construction uses explicit formulas, and those lose significance when dealing with objects up to homotopy.
\end{rem}

\subsection{A conjecture} \label{subsec:conjectures}

In \cite[Remark 1.8]{Porta_GAGA_2015} we referred to a notion of stable Morita equivalence.
We report the definition here:

\begin{defin}
	Let $\varphi \colon \cT' \to \cT$ be a morphism of pregeometries. We will say that $\varphi$ is a \emph{(hypercomplete) stable Morita equivalence} if for every (hypercomplete) $\infty$-topos $\cX$ and every $\cO \in \Str_{\cT}(\cX)$ the functor \eqref{eq:Phi_stable} is an equivalence of $\infty$-categories.
\end{defin}

Observe that \cref{cor:first_reduction_step} can easily be transformed into a reduction statement for hypercomplete stable Morita equivalences.

\begin{lem} \label{lem:semi_discrete_devissage}
	Let $\varphi \colon \cT' \to \cT$ be a morphism of pregeometries.
	\begin{enumerate}
		\item if $\varphi$ is a weak Morita equivalence, then it is also a stable Morita equivalence;
		\item if the induced functor $\varphi_{\mathrm{sd}} \colon \cT'_{\mathrm{sd}} \to \cT_{\mathrm{sd}}$ is a stable Morita equivalence, then the same goes for $\varphi$.
	\end{enumerate}
\end{lem}

\begin{proof}
	Assertion (1) follows immediately from the definitions.
	Assertion (2) is a simple reformulation of \cite[Corollary 1.13]{Porta_GAGA_2015}.
\end{proof}

We formulate the following conjecture:

\begin{conj} \label{conj:equivalence_modules}
	The morphism $\cTdisc \to \Tan$ is a hypercomplete stable Morita equivalence of pregeometries.
\end{conj}

The following result hints that the conjecture could be true:

\begin{prop} \label{prop:heart_equivalence_of_modules}
	Let $\cX$ be an $\infty$-topos with enough points and let $\cO \in \Strloc_{\Tan}(\cX)$.
	The adjunction
	\[ \Phi_\cO \colon \Sp(\Strloc_{\Tan}(\cX)_{/ \cO}) \rightleftarrows \Sp(\Strloc_{\cTdisc}(\cX)_{/ \cO\alg}) \colon \Psi_\cO  \]
	restricts to an equivalence on the hearts of the respective $t$-structures.
\end{prop}

\begin{proof}
	This is an immediate consequence of \cref{cor:reduction_to_spaces}, \cref{prop:t_structure_Str}.(5) and \cref{prop:equivalence_discrete_modules}.
\end{proof}

The difference with \cref{thm:equivalence_of_modules} is obviously that we no longer require the $\Tan$-structured topos $(\cX, \cO_\cX)$ to be a derived \canal space. It is worth of noting that this assumption was used at only one point of the proof: namely, in the proof of \cref{lem:elementary_flatness}.

\section{Postnikov towers} \label{sec:Postnikov}

In this section we discuss the first important application of \cref{thm:equivalence_of_modules}, namely the structure theorem for square-zero extensions in the \canal setting.
Let $(\cX, \cO_\cX)$ be a derived \canal space.
The cited theorem exhibits an equivalence of $\infty$-categories
\[ \Psi \colon \Sp\left(\Strloc_{\cTdisc}(\cX)_{/ \cO_\cX\alg}\right) \rightleftarrows \Sp\left(\Strloc_{\Tan}(\cX)_{/ \cO_\cX}\right) \colon \Phi \]
To ease notations, in this section we will systematically suppress the subscript $\cO_\cX$.
Moreover, \cref{cor:Psi_preserves_spectra} shows that both the squares
\[ \begin{tikzcd}
\Sp\left(\Strloc_{\cTdisc}(\cX)_{/\cO_\cX\alg}\right) \arrow[yshift = 0.60ex]{r}{\Psi} \arrow{d}{\Omega^\infty_{\mathrm{alg}}} & \Sp\left(\Strloc_{\Tan}(\cX)_{/ \cO_\cX}\right) \arrow{d}{\Omega^\infty_{\mathrm{an}}} \arrow[yshift = -0.60ex]{l}{\Phi} \\
\Strloc_{\cTdisc}(\cX)_{\cO_\cX\alg // \cO_\cX\alg} \arrow[yshift = 0.60ex]{r}{\overline{\Psi}} & \Strloc_{\Tan}(\cX)_{\cO_\cX // \cO_\cX} \arrow[yshift = -0.60ex]{l}{\overline{\Phi}}
\end{tikzcd} \]
commute.
We could therefore avoid making any distinction between $\cO_\cX\alg \textrm{-} \Mod \simeq \Sp(\Strloc_{\cTdisc}(\cX)_{/ \cO_\cX\alg})$ and $\Sp(\Strloc_{\Tan}(\cX)_{/\cO_\cX})$.
However, we prefer to emphasize the different nature of the two categories and therefore we will maintain the two notations.

In addition, even though we will rarely need to work with more general $\Tan$-structured topoi than derived \canal spaces, it still seems worth of to develop the theory of this section in the greatest possible generality. In this case, the role of \cref{thm:equivalence_of_modules} would be played by \cref{conj:equivalence_modules}, and therefore we won't assume that $\Phi_\cO$ is an equivalence. We will see that \cref{prop:heart_equivalence_of_modules} is largely enough for our purposes.

\begin{defin} \label{def:analytic_derivation}
	Let $\cX$ be an $\infty$-topos and let $f \colon A \to B$ be a morphism in $\Strloc_{\Tan}(\cX)$.
	Let $M$ be an analytic $B$-module.
	An \emph{$A$-linear analytic derivation of $B$ with values in $M$} is a morphism $B \to \Omega^\infty_{\mathrm{an}}(M)$ in the $\infty$-category $\Strloc_{\cT}(\cX)_{A // B}$.
\end{defin}

\begin{defin} \label{def:analytic_square_zero_extension}
	Let $\cX$ be an $\infty$-topos and let $f \colon A \to B$ be a morphism in $\Strloc_{\Tan}(\cX)$.
	We will say that $f$ is \emph{an analytic square-zero extension} if there exists an analytic $B$-module $M$ and an $A$-linear derivation $d \colon B \to \Omega^\infty_{\mathrm{an}}(M)$ such that the square
	\[ \begin{tikzcd}
	A \arrow{r} \arrow{d}{f} & B \arrow{d}{d} \\
	B \arrow{r}{d_0} & \Omega^\infty_{\mathrm{an}}(M)
	\end{tikzcd} \]
	is a pullback in $\Strloc_{\Tan}(\cX)_{/B}$.
	In this case, we will say that $f$ is a square-zero extension of $B$ by $M[-1]$.
\end{defin}

As we explained in the introduction, knowing that a morphism is a square-zero extension gives a lot of information on the morphism itself.
This is often used in practice in order to reduce derived statements to underived ones.
To do this, however, we will need to know that whenever $(\cX, \cO_\cX)$ is a derived \canal space, the canonical morphisms $\tau_{\le n} \cO_\cX \to \tau_{\le n - 1} \cO_\cX$ are square-zero extensions.
In derived algebraic geometry the proof of this statement passes through the so-called ``structure theorem for square-zero extension'', which provides a computational recognizing criterion for square-zero extensions.
We will proceed in the same way in this \canal setting.

\begin{defin} \label{def:analytic_small_extension}
	Let $\cX$ be an $\infty$-topos and let $f \colon A \to B$ be a morphism in $\Strloc_{\Tan}(\cX)$.
	We will say that $f$ is an $n$-small extension if the morphism $f\alg \colon A\alg \to B\alg$ is an $n$-small extension in the sense of \cite[Definition 7.4.1.18]{Lurie_Higher_algebra}.
\end{defin}

\begin{rem}
	In other words, a morphism $f \colon A \to B$ in $\Strloc_{\Tan}(\cX)$ is an $n$-small extension if the following two conditions are satisfied:
	\begin{enumerate}
		\item The fiber $\mathrm{fib}(f\alg)$ is $n$-connective and $(2n)$-truncated, and
		\item the multiplication map $\mathrm{fib}(f\alg) \otimes_{A\alg} \mathrm{fib}(f\alg) \to \mathrm{fib}(f\alg)$ is nullhomotopic.
	\end{enumerate}
	This is, in some sense, a computational criterion: see \cite[7.4.1.20]{Lurie_Higher_algebra}.
\end{rem}

We can now state the main result of this section:

\begin{thm} \label{thm:partial_equivalence_sqzero_small}
	Let $\cX$ be an $\infty$-topos and let $f \colon A \to B$ be a morphism in $\Strloc_{\Tan}(\cX)$.
	Suppose that the following conditions hold:
	\begin{enumerate}
		\item the morphism $f$ is an $n$-small extension;
		\item there exists an analytic module $M \in \Sp(\Strloc_{\Tan}(\cX)_{/B})$ and an $A\alg$-derivation $d \colon B\alg \to \Omega^\infty_{\mathrm{alg}}(\Phi(M))$ fitting in a pullback square
		\[ \begin{tikzcd}
		A\alg \arrow{r}{f\alg} \arrow{d}[swap]{f\alg} & B\alg \arrow{d}{d} \\
		B\alg \arrow{r}{d_0\alg} & \Omega^\infty_{\mathrm{alg}}(\Phi(M))
		\end{tikzcd} \]
		in $\Strloc_{\cTdisc}(\cX)_{/ B\alg}$.
	\end{enumerate}
	Then $f$ is an analytic square-zero extension of $B$ by $M[-1]$.
\end{thm}

The proof is somehow technical. Before discussing it, let us derive the main consequences of this result.

\begin{cor} \label{cor:truncations_are_analytic_square_zero}
	Let $\cX$ be an $\infty$-topos.
	Let $\cO \in \Str_{\Tan}(\cX)$.
	For every non-negative integer $n$, the natural map $\tau_{\le n} \cO \to \tau_{\le n - 1} \cO$ is a square-zero extension.
\end{cor}

\begin{proof}
	In virtue of \cref{thm:partial_equivalence_sqzero_small}, it will be enough to check that for every non-negative integer $n$ there exists an analytic module $\cF_n \in \Sp(\Strloc_{\Tan}(\cX)_{/\tau_{\le n - 1} \cO})$ such that $\Phi(\cF_n) \simeq \pi_n(\cO\alg)[n+1]$. Since $\pi_n(\cO\alg)$ is a discrete $\tau_{\le n-1}(\cO)$-module, the result follows immediately from \cref{prop:heart_equivalence_of_modules}.
\end{proof}

When $(\cX, \cO_\cX)$ is a derived \canal space, \cref{thm:equivalence_of_modules} allows us to deduce a much stronger consequence:

\begin{cor} \label{cor:structure_theorem_square_zero}
	Let $(\cX, \cO_\cX)$ be a derived \canal space.
	Let $\cO \in \Str_{\Tan}(\cX)$ be any other analytic structure on $\cX$ and let $f \colon \cO \to \cO_\cX$ be a local morphism.
	Fix furthermore a non-negative integer $n$.
	Then the following conditions are equivalent:
	\begin{enumerate}
		\item $f$ is an analytic square-zero extension by $M[-1]$ and $\Omega^\infty_{\mathrm{an}}(M)$ is $n$-connective and $(2n)$-truncated;
		\item $f$ is an $n$-small extension.
	\end{enumerate}
\end{cor}

\begin{proof}
	If $f$ is a square-zero extension by $M[-1]$ then $f\alg$ is a square-zero extension by $\Phi(M)[-1]$.
	Since $\Phi$ is $t$-exact, we deduce that (2) holds from the structure theorem for algebraic square-zero extensions, see \cite[7.4.1.26]{Lurie_Higher_algebra}.
	
	Vice-versa, suppose that $f$ is an $n$-small extension. Then there exists $N \in \cO_\cX\alg \textrm{-} \Mod$ such that $\Omega^\infty_{\mathrm{an}}(N)$ is $n$-connective and $(2n)$-truncated and there exists a $\cO\alg$-linear derivation $d \colon \cO_\cX\alg \to \Omega^\infty_{\mathrm{an}}(N)$ fitting in a pullback square
	\[ \begin{tikzcd}
	\cO\alg \arrow{d}[swap]{f\alg} \arrow{r}{f\alg} & \cO_\cX\alg \arrow{d}{d} \\
	\cO_\cX\alg \arrow{r}{d_0\alg} & \Omega^\infty_{\mathrm{alg}}(N)
	\end{tikzcd} \]
	Since $(\cX, \cO_\cX)$ is a derived \canal space, we can invoke \cref{thm:equivalence_of_modules} to find an analytic module $M \in \Sp(\Strloc_{\Tan}(\cX)_{/\cO_\cX})$ such that $\Phi(M) = N$.
	In particular, we see that $\Omega^\infty_{\mathrm{alg}}(\Phi(M)) = \overline{\Phi}(\Omega^\infty_{\mathrm{an}}(M))$ is $n$-connective and $(2n)$-truncated, and therefore the same goes for $\Omega^\infty_{\mathrm{an}}(M)$.
	Thus, the hypotheses of \cref{thm:partial_equivalence_sqzero_small} are satisfied and as consequence we conclude that $f$ is an analytic square-zero extension.
\end{proof}

\begin{rem}
	\cref{cor:structure_theorem_square_zero} is saying that as soon as we know that a given map is analytic, its infinitesimal properties are determined algebraically.
	\cref{thm:equivalence_of_modules} can be already seen as an instance of this phenomenon, and we will encounter other examples in \cite{Porta_Cotangent_2015}.
\end{rem}

\cref{thm:partial_equivalence_sqzero_small} is a direct consequence of the following general ``analytification result'':

\begin{prop} \label{prop:lifting_derivations}
	Let $\cX$ be an $\infty$-topos and let $A, B, C \in \Str_{\Tan}(\cX)$.
	Let $d_0 \colon A \to C$, $p \colon B \to A$ be morphisms in $\Strloc_{\Tan}(\cX)$, and let $d \colon A\alg \to C\alg$ be such that the square
	\begin{equation} \label{eq:algebraic_derivation}
	\begin{tikzcd}
	B\alg \arrow{r}{p\alg} \arrow{d}[swap]{p\alg} & A\alg \arrow{d}{d} \\
	A\alg \arrow{r}{d_0\alg} & C\alg
	\end{tikzcd}
	\end{equation}
	is a pullback in $\Strloc_{\cTdisc}(\cX)_{/C}$.
	Suppose furthermore that $d_0$ is an effective epimorphism and that it admits a retraction $\pi \colon C \to A$ and that $\pi\alg$ is a retraction for $d$ as well.
	Then there exists a morphism $\overline{d} \colon A \to C$ in $\Strloc_{\Tan}(\cX)$ such that the square
	\[ \begin{tikzcd}
	B \arrow{r}{p} \arrow{d}[swap]{p} & A \arrow{d}{\overline{d}} \\
	A \arrow{r}{d_0} & C
	\end{tikzcd} \]
	is a pullback. Moreover, $\overline{d}\alg = d$.
\end{prop}

The rest of this section will be devoted to give a proof of this proposition.
First, let us state a couple of elementary lemmas.

\begin{lem}
	Let $\mathbf \Delta_s$ be the semisimplicial category and let $\mathbf \Delta_{s, \le 1}$ be the full subcategory of $\mathbf \Delta_s$ spanned by the objects $[0]$ and $[1]$.
	Then for any $n \ge 2$ the category $(\mathbf \Delta_{s, \le 1})_{/[n]}$ is connected.
\end{lem}

\begin{proof}
	First of all, let us observe that every object of the form $[0] \to [n]$ in $(\mathbf \Delta_{s, \le 1})_{/[n]}$ is connected to some object of the form $[1] \to [n]$.
	Indeed, if $[0] \to [n]$ picks the element $k$ and $k < n$, it is sufficient to consider $[1] \to [n]$ defined by $0 \mapsto k$, $1 \mapsto n$ and $d_1 \colon [0] \to [1]$ is a morphism in $(\mathbf \Delta_{s, \le 1})_{/[n]}$.
	If, instead $k = n$, it is sufficient to consider $[1] \to [n]$ defined by $0 \mapsto 0$ and $1 \mapsto n$, so that $d_0 \colon [0] \to [1]$ is a morphism in $(\mathbf \Delta_{s, \le 1})_{/[n]}$.
	It will be therefore sufficient to prove that any two morphisms $f,g \colon [1] \to [n]$ can be connected by a zig-zag of morphisms in $(\mathbf \Delta_{s, \le 1})_{/[n]}$.
	For ease of notation, let us denote by $f_{a,b}$ (with $0 \le a < b \le n$) the morphism $[1] \to [n]$ defined by $0 \mapsto a$, $1 \mapsto b$.
	Suppose $0 < b < n$. Then, for every $a < b$, $f_{a,b}$ is connected to $f_{a, b+1}$. Indeed, we see that $f_{a,b} \circ d_1 = f_{a,b+1} \circ d_1$.
	Similarly, if $0 < a < b$, then $f_{a-1,b}$ is connected to $f_{a,b}$. Indeed, $f_{a-1,b} \circ d_0 = f_{a,b} \circ d_0$.
	Therefore any $f_{a,b}$ is connected to $f_{a,n}$, and every $f_{a,n}$ is connected to $f_{0,n}$.
\end{proof}

\begin{lem} \label{lem:semisimplicial_I}
	Let $\cC$ be an $\infty$-category with finite limits. Let $X \in \cC$ be an object.
	Let $U \colon \cC_{/X} \to \cC$ be the forgetful functor.
	Let $F \colon \mathbf \Delta_{s, \le 1} \to \cC_{/X}$ be any functor and let $j_1 \colon \mathbf \Delta_{s, \le 1} \to \mathbf \Delta_s$ be the natural inclusion.
	Then $U \circ \mathrm{Ran}_{j_1}(F) \simeq \mathrm{Ran}_{j_1}(U \circ F)$.
\end{lem}

\begin{proof}
	The forgetful functor $U$ commutes with connected limits, and by the previous lemma, the limits involved in $\mathrm{Ran}_{j_1}(F)$ are connected.
\end{proof}

\begin{lem} \label{lem:semisimplicial_II}
	Let $\cC$ be an $\infty$-category with finite limits.
	Let $f \colon Y \to X$ be any morphism and let $Y^\bullet_s \colon \mathbf \Delta_s \to \cC$ be the underlying semi-simplicial object underlying the \v{C}ech nerve $\check{\cC}(f)$.
	Then $Y^\bullet_s \simeq \mathrm{Ran}_{j_1}(Y^\bullet_s \circ j_1)$.
\end{lem}

\begin{proof}
	Let $U \colon \cC_{/X} \to \cC$ be the forgetful functor.
	Seeing $f$ as a functor $f \colon \Delta^0 \to \cC_{/X}$, we see that
	\[ Y^\bullet_s \simeq U \circ \mathrm{Ran}_{j_0}(f) \]
	where $j_0 \colon \Delta^0 = \mathbf \Delta_{s, \le 0} \to \mathbf \Delta_s$ is the natural inclusion.
	Let us set $V^\bullet \coloneqq \mathrm{Ran}_{j_0}(f)$.
	We claim that $V^\bullet \simeq \mathrm{Ran}_{j_1}(V^\bullet \circ j_1)$.
	Assuming this, the previous lemma implies that
	\begin{align*}
	Y^\bullet_s & = U \circ V^\bullet \simeq U \circ \mathrm{Ran}_{j_1}(V^\bullet \circ j_1) \\
	& \simeq \mathrm{Ran}_{j_1}(U \circ V^\bullet \circ j_1) \\
	& = \mathrm{Ran}_{j_1}(Y^\bullet_s \circ j_1)
	\end{align*}
	We are left to prove the claim.
	We can factor $j_0$ as
	\[ \begin{tikzcd}
	\mathbf \Delta_{s, \le 0} \arrow{r}{j_{01}} & \mathbf \Delta_{s, \le 1} \arrow{r}{j_1} & \mathbf \Delta_s
	\end{tikzcd} \]
	Therefore we have
	\[ \mathrm{Ran}_{j_0}(f) = \mathrm{Ran}_{j_1}(\mathrm{Ran}_{j_{01}}(f)) \]
	On the other side, since $j_1$ is fully faithful, we see that
	\[ \mathrm{Ran}_{j_0}(f) \circ j_1 = \mathrm{Ran}_{j_1}(\mathrm{Ran}_{j_{01}}(f)) \circ j_1 = \mathrm{Ran}_{j_{01}}(f) \]
	Therefore
	\begin{align*}
	\mathrm{Ran}_{j_1}(V^\bullet \circ j_1) & = \mathrm{Ran}_{j_1}(\mathrm{Ran}_{j_0}(f) \circ j_1) \\
	& = \mathrm{Ran}_{j_1}(\mathrm{Ran}_{j_{01}}(f)) \\
	& = \mathrm{Ran}_{j_0}(f) = V^\bullet
	\end{align*}
\end{proof}

\begin{lem} \label{lem:reflexive_pullback}
	Let $\cC$ be an $\infty$-category with pullbacks.
	Suppose given a morphism $f \colon X \to Y$ together with a retraction $g \colon Y \to X$.
	Then, up to replacing $f$ with an equivalent morphism, there exists a pullback diagram of the form
	\[ \begin{tikzcd}
	Z \arrow{r}{p} \arrow{d}{p} & X \arrow{d}{f} \\
	X \arrow{r}{f} & Y
	\end{tikzcd} \]
\end{lem}

\begin{proof}
	The forgetful functor $\cC_{/X} \to \cC$ preserves connected limits.
	Therefore, it will be enough to prove the statement in $\cC_{/X}$.
	In this case, (the identity of) $X$ is a final object.
	Applying \cite[2.3.3.8]{HTT}, \personal{With $S = \Delta^0$} we can find a full subsimplicial set $\cD \subset \cC_{/X}$ which is a categorical equivalence and such that $\cD$ is a minimal $\infty$-category.
	Up to replacing $f$ with an equivalent morphism, we can suppose that $f$ belongs to $\cD$.
	Therefore, in $\cD$ we can form the pullback square
	\[ \begin{tikzcd}
	Z \arrow{r}{p_1} \arrow{d}{p_2} & X \arrow{d}{f} \\
	X \arrow{r}{f} & Y
	\end{tikzcd} \]
	Since $X$ is a final object, the two morphisms $p_1, p_2 \colon Z \to X$ are homotopic relative to $\partial \Delta^1$.
	Since $\cD$ is a minimal $\infty$-category, we conclude that $p_1 = p_2$.
\end{proof}

We are now ready for the proof of \cref{prop:lifting_derivations} and therefore for the one of \cref{thm:partial_equivalence_sqzero_small}.

\begin{proof}[Proof of \cref{prop:lifting_derivations}]
	Given a morphism $g \colon X \to Y$ in an $\infty$-category with pullbacks $\cC$, we will denote by $\Cech_s(p)$ the semi-simplicial object underlying the \v{C}ech nerve of $p$.
	Using the retraction $\pi \colon C \to A$ to work in $\Strloc_{\Tan}(\cX)_{/A}$, we can apply \cite[Corollary 1.18]{Porta_GAGA_2015} to deduce that there are canonical equivalences of semi-simplicial objects:
	\[ \Cech_s(p)\alg \simeq \Cech_s(p\alg), \qquad \Cech_s(d_0)\alg \simeq \Cech_s(d_0\alg) \]
	Since $d_0$ is an effective epimorphism, invoking \cite[6.5.3.7]{HTT} we conclude that
	\[ |\Cech_s(d_0) | \simeq C , \]
	where the geometric realization is taken in $\Strloc_{\Tan}(\cX)_{/A}$.
	Moreover, the morphism $p\alg$ is an effective epimorphism (being the pullback of $d_0\alg$ by hypothesis). Therefore, we deduce from \cite[Lemma 11.11]{DAG-IX} that $p$ is an effective epimorphism as well. In other words, in $\Strloc_{\Tan}(\cX)_{/A}$ we also have an equivalence:
	\[ |\Cech_s(p) | \simeq A \]
	Suppose we are able to construct a map of semi-simplicial objects $\varphi \colon \Cech_s(p) \to \Cech_s(d_0)$ in $\Strloc_{\Tan}(\cX)_{/A}$.
	We would then obtain a morphism $\overline{d} \coloneqq | \varphi | \colon A \to C$. This map would naturally fit in a commutative diagram
	\[ \begin{tikzcd}
	B \arrow{r}{p} \arrow{d}{p} & A \arrow{d}{\overline{d}} \arrow[bend left]{ddr}{\mathrm{id}} \\
	A \arrow{r}{d_0} \arrow[bend right]{drr}{\mathrm{id}} & C \arrow{dr}{\pi} \\
	& & A
	\end{tikzcd} \]
	Using again \cite[Corollary 1.18]{Porta_GAGA_2015}, we would deduce that the image of this commutative square under the functor $(-)\alg$ is precisely the pullback diagram \eqref{eq:algebraic_derivation}.
	Since $(-)\alg \colon \Strloc_{\Tan}(\cX)_{/A} \to \Strloc_{\cTdisc}(\cX)_{/A\alg}$ is conservative and commutes with pullbacks, we conclude that the original square was a pullback as well, thus completing the proof.
	
	We are therefore left to exhibit an analytic lifting of the canonical map $\Cech_s(p\alg) \to \Cech_s(d_0\alg)$.
	Define $A_1$ to be the pullback
	\[ \begin{tikzcd}
	A_1 \arrow{r}{q_1} \arrow{d}{q_2} & A \arrow{d}{d_0} \\
	A \arrow{r}{d_0} & C
	\end{tikzcd} \]
	Since $d_0$ admits a retraction $\pi \colon C \to A$, we can apply \cref{lem:reflexive_pullback} and write, without loss of generality, $q_1 = q = q_2$.
	Now define $B_1$ and $B_1'$ to be the pullbacks
	\[ \begin{tikzcd}
	B_1 \arrow{r}{t} \arrow{d}{f} & B \arrow{d}{p} \\
	A_1 \arrow{r}{q} & A
	\end{tikzcd} \qquad
	\begin{tikzcd}
	B_1' \arrow{r}{t_1} \arrow{d}{t_2} & B \arrow{d}{p} \\
	B \arrow{r}{p} & A
	\end{tikzcd} \]
	We have a canonical commutative diagram
	\[ \begin{tikzcd}
	B_1 \arrow{r}{t} \arrow{d}{t} & B \arrow{d}{p} \\
	B \arrow{r}{p} & A
	\end{tikzcd} \]
	that induces a canonical map $h \colon B_1 \to B_1'$ making the diagram
	\[ \begin{tikzcd}
	{} & B_1 \arrow{d}{h} \arrow{dr}{t} \arrow{dl}[swap]{t} \\
	B & B_1' \arrow{l}[swap]{t_1} \arrow{r}{t_2} & B
	\end{tikzcd} \]
	commutative.	
	In $\Strloc_{\cTdisc}(\cX)$ we have a commutative cube
	\[ \begin{tikzcd}
	B_1\alg \arrow{dr}{h\alg} \arrow{dd}[swap]{t\alg} \arrow[bend left = 10]{drrr}{t\alg} \arrow{dddr}[yshift = -17, xshift = 7]{f\alg} \\
	{} & (B_1')\alg \arrow[dashed]{dd} \arrow{rr}{t_1\alg} \arrow[crossing over]{dl}[swap, near end]{t_2\alg} & & B\alg \arrow{dd}{p\alg} \arrow{dl}{p\alg} \\
	B\alg \arrow[crossing over]{rr}[xshift = 15]{p\alg} \arrow{dd}[swap]{p\alg} & & A\alg \\
	{} & A_1\alg \arrow{rr}[xshift = -10]{q\alg} \arrow{dl}[swap, near start]{q\alg} & & A\alg \arrow{dl}{d_0\alg} \\
	A\alg \arrow{rr}{d_0\alg} & & C\alg \arrow[leftarrow, crossing over]{uu}[near end, swap]{d}
	\end{tikzcd} \]
	Since the bottom square is a pullback, we deduce the existence of the dotted arrow as well as the commutativity of the triangle
	\[ \begin{tikzcd}
	B_1\alg \arrow{r}{h\alg} \arrow{dr}[swap]{f\alg} & (B_1')\alg \arrow[dashed]{d} \\
	{} & A_1\alg
	\end{tikzcd} \]
	Consider now the rectangle
	\[ \begin{tikzcd}
	(B_1')\alg \arrow{r}{t_1\alg} \arrow{d}{t_2\alg} & B\alg \arrow{d}{p\alg} \arrow{r}{p\alg} & A\alg \arrow{d}{d_0\alg} \\
	B\alg \arrow{r}{p\alg} & A\alg \arrow{r}{d\alg} & C\alg .
	\end{tikzcd} \]
	Since both the squares are pullbacks, the same goes for the outer rectangle.
	As the right square in the diagram
	\[ \begin{tikzcd}
	(B_1')\alg \arrow[dashed]{r} \arrow{d}{t_2\alg} & A_1\alg \arrow{r}{q\alg} \arrow{d}{q\alg} & A\alg \arrow{d}{d_0\alg} \\
	B_1\alg \arrow{r}{p\alg} & A\alg \arrow{r}{d_0\alg} & C\alg
	\end{tikzcd} \]
	is a pullback, we conclude that the same goes for the left one.
	Since $(-)\alg$ preserves pullback square and is conservative, we first conclude that $h\alg$ is an equivalence, hence that $h$ is an equivalence as well.
	We therefore obtain a well defined morphism of $1$-truncated semi-simplicial objects.
	Lemmas \ref{lem:semisimplicial_I} and \ref{lem:semisimplicial_II} show that this is enough to conclude.
\end{proof}

\appendix

\section{A relative flatness result} \label{sec:relative_analytification}

Let $X = (\cX, \cO_\cX)$ be a derived \canal space.
Let $x\inv \colon \cX \to \cS \colon x_*$ be a geometric point of the $\infty$-topos $\cX$ and set $\cO \coloneqq x\inv \circ \cO_\cX \in \Str_{\Tan}(\cS)$.
Composition with the morphism of pregeometries $\varphi \colon \cTzar \to \cTdisc$ induces a forgetful functor
\[ \overline{\Phi}_\cO \colon \Strloc_{\Tan}(\cS)_{\cO /} \to \Strloc_{\cTzar}(\cS)_{\cO\alg /} . \]
We already remarked at the beginning of \cref{subsec:stable_Morita} that this functor admits a left adjoint $\overline{\Psi}_\cO$.
As in \cite[§6.5]{Porta_GAGA_2015}, we introduce a couple of special notations.
First of all, we will denote by $(-)\an$ the composite $\overline{\Phi}_\cO \circ \overline{\Psi}_\cO$.
There are forgetful functors
\[ \Strloc_{\Tan}(\cS)_{\cO /} \to \cS, \qquad \Strloc_{\cTzar}(\cS)_{\cO\alg /} \to \cS, \]
and both admits left adjoints. We will denote them by $\cO\{-\}$ and $\cO\alg[-]$ respectively.
Moreover, we will denote by $\cO\alg \{-\}$ the composite $\overline{\Phi}_\cO \circ \cO\{-\}$.
Finally, we observe that $\cO\{-\} \simeq \overline{\Psi}_\cO \circ \cO\alg[-]$.
With these notations, we can state the main result of this section, which is a relative version of \cite[Theorem 6.19]{Porta_GAGA_2015}

\begin{prop} \label{prop:relative_flatness}
	Let $A \in \Strloc_{\cTzar}(\cS)_{\cO\alg /}$.
	The canonical map $A \to A\an$ is flat (in the derived sense).
\end{prop}

The strategy of the proof is the same we used in \cite{Porta_GAGA_2015}.
However, the initial computations have to be discussed again.

\begin{lem}[{Cf.\ \cite[Proposition 2.37]{Porta_GAGA_2015}}] \label{lem:relative_computation_I}
	Let us denote by $\Delta^0$ the one point space.
	Then
	\[ \pi_i ( \cO\alg\{\Delta^0\} ) = \begin{cases} \pi_0(\cO\alg)\{z\} & \text{if } i = 0 \\ \pi_i(\cO\alg) \otimes_{\pi_0(\cO\alg)} \pi_0(\cO\alg)\{z\} & \text{otherwise.}
	\end{cases} \]
	where we denoted by $\pi_0(\cO\alg)\{z\}$ the underlying algebra of the discrete $\Tan$-structure $\pi_0(\cO\alg) \cotimes_{\mathbb C} \cH_1$ (see \cite[§2.2, §2.3]{Porta_GAGA_2015} for the notations).
\end{lem}

\begin{proof}
	We have a canonical identification $\cO\{\Delta^0\} \simeq \cO \cotimes_{\mathbb C} \cH\{ \Delta^0 \}$, where $\cH\{-\}$ is the functor introduced in \cite[§2.5]{Porta_GAGA_2015} and $\cotimes$ denote the coproduct in the $\infty$-category $\Strloc_{\Tan}(\cS)$ (which exists in virtue of \cite[Corollary 2.3]{Porta_GAGA_2015}).
	Now, $\pi_0$ is a left adjoint and therefore it commutes with coproducts. Thus we dealt with the case $i = 0$.
	
	Let now $i > 0$. Recall that $(\cS, \cO)$ was the germ of a derived \canal space $X = (\cX, \cO_\cX)$ at the geometric point $x\inv \colon \cX \rightleftarrows \cS \colon x_*$.
	We can assume without loss of generality that $\cX$ is $0$-localic and that we can find an open (underived) Stein space $U$ and a closed immersion $j \colon X \to \Spec^{\Tan}(U)$ (see \cite[Lemma 12.13]{DAG-IX}).
	We will denote by $x$ the point of the underlying topological space of $X$ corresponding to the geometric point $x\inv$. Moreover, we will still denote by $x$ the image of this point in $U$ via the closed immersion $j$.
	
	We can identify $(\cS, \cH\{\Delta^0\})$ with the germ of the analytic affine line $\cE^1_{\mathbb C} = \Spec^{\Tan}(\mathbb C)$ at the origin.
	We have $\Spec^{\Tan}(U) \times \cE^1_{\mathbb C} \simeq \Spec^{\Tan}(U \times \mathbb C)$, and therefore this is a $0$-truncated derived \canal space.
	In particular, its germ at the point $(x,0)$ is $0$-truncated. This germ is $\cO_{U,x} \cotimes_{\mathbb C} \cH\{\Delta^0\}$.
	Reasoning as in \cite[Lemma C.4]{Porta_Yu_Higher_analytic_stacks_2014} \personal{but invoking Demailly instead of Douady} we conclude that the morphism
	\begin{equation} \label{eq:computation_relative_flatness}
	\cO_{U,x}\alg \to (\cO_{U,x} \cotimes_{\mathbb C} \cH\{\Delta^0\} )\alg
	\end{equation}
	is flat (in the usual sense of commutative algebra, since both are discrete $\mathbb E_\infty$-rings).
	By construction, there exists an effective epimorphism $\cO_{U,x} \to \cO$, and we can identify the germ of $X \times \cE^1_{\mathbb C}$ at $(x,0)$ (which is $\cO\{\Delta^0\}$) with the coproduct
	\[ \begin{tikzcd}
	\cO_{U,x} \arrow{r} \arrow{d} & \cO_U \cotimes_{\mathbb C} \cH \{\Delta^0\} \arrow{d} \\
	\cO \arrow{r} & \cO\{\Delta^0\}
	\end{tikzcd} \]
	computed in the $\infty$-category $\Strloc_{\Tan}(\cS)$.
	Since the map $\cO_{U,x} \to \cO$ is an effective epimorphism, the functor $\overline{\Phi} = (-)\alg$ preserves this coproduct.
	Therefore, the map $\cO \to \cO\{ \Delta^0\}$ is flat in the derived sense.
	The conclusion follows.
\end{proof}

\begin{lem}[{Cf.\ \cite[Proposition 2.38]{Porta_GAGA_2015}}] \label{lem:relative_computation_II}
	For every $n \ge 2$, we have $\pi_0(\cO\alg\{\partial \Delta^n\}) \simeq \pi_0(\cO\alg)\{z\}$.
\end{lem}

\begin{proof}
	Since $(-)\alg$ commutes with $\pi_0$ by definition, it will be enough to show that $\pi_0(\cO\{ \partial \Delta^n\}) \simeq \pi_0(\cO) \cotimes_{\mathbb C} \cH_1$.
	Since $\pi_0$ commutes with coproducts we have $\pi_0(\cO \{\partial \Delta^n\}) \simeq \pi_0(\cO \cotimes_{\mathbb C} \cH\{\partial \Delta^n\} ) \simeq \pi_0(\cO) \cotimes_{\mathbb C} \pi_0(\cH \{\partial \Delta^n\} )$, and now the result follows from \cite[Proposition 2.38]{Porta_GAGA_2015}.
\end{proof}

\begin{lem} \label{lem:elementary_flatness}
	The morphisms $\cO[\Delta^0] \to \cO\{\Delta^0\}$ and $\cO[\partial \Delta^n] \to \cO\{\partial \Delta^n\}$ (for $n \ge 0$) are flat.
\end{lem}

\begin{proof}
	The same proof of \cite[Lemma 6.17]{Porta_GAGA_2015} shows that we can deduce the second statement from the first.
	On the other side, we know that the canonical morphism $\cO \to \cO[\Delta^0]$ is flat. Combining this with \cref{lem:relative_computation_I}, we see that we only need to show that $\pi_0( \cO[\Delta^0] ) \to \pi_0( \cO\{\Delta^0\} )$ is flat. This has already been discussed around \eqref{eq:computation_relative_flatness}.
\end{proof}

\begin{lem}
	The diagram
	\[ \begin{tikzcd}
	\cO\alg[\partial \Delta^n] \arrow{r} \arrow{d} & \cO\alg[\Delta^0] \arrow{d} \\
	\cO\alg\{\partial \Delta^n\} \arrow{r} & \cO\alg\{\Delta^0\}
	\end{tikzcd} \]
	is a pushout in the $\infty$-category $\Strloc_{\cTzar}(\cS)_{\cO\alg // \cO\alg}$.
\end{lem}

\begin{proof}
	Pushouts in $\Strloc_{\cTzar}(\cS)_{\cO\alg /}$  are simply tensor products of $\mathbb E_\infty$-rings.
	Let
	\[ R \coloneqq \cO\alg\{\partial \Delta^n\} \otimes_{\cO\alg[\partial \Delta^n]} \cO\alg[\Delta^0] \]
	\Cref{lem:elementary_flatness} shows that both the maps $\cO\alg[\Delta^0] \to R$ and $\cO\alg[\Delta^0] \to \cO\alg\{\Delta^0\}$ are flat.
	Therefore, the canonical map $g \colon R \to \cO\alg\{\Delta^0\}$ is flat as well.
	It will therefore be sufficient to show that $\pi_0(g)$ is an isomorphism.
	However, \cref{lem:relative_computation_II} shows that $\pi_0(R) \simeq \pi_0(\cO\alg\{\partial \Delta^n\}) \simeq \pi_0(\cO\alg\{\Delta^0\})$.
	The proof is therefore completed.
\end{proof}  

At this point the proof of \cref{prop:relative_flatness} proceeds as the one of \cite[Theorem 6.19]{Porta_GAGA_2015}.

\begin{rem}
	The results of this section fit in the slightly more general setting of ``algebraic geometry relative to an analytic base''. The first foundational reference on the subject was Hakim's thesis \cite{Hakim_Topos_1972}. It is possible to deal with such a theory using the language of pregeometries. We will come back to this subject in a subsequent work.
\end{rem}

\section{On left adjointable squares} \label{sec:left_adjointable}

\subsection{The Beck-Chevalley condition}

We collect in this final section some material on left adjointable squares we weren't able to find in the literature.
Let us start by reviewing such notion.
Let $\sigma$:
\begin{equation} \label{eq:Beck_Chevalley_situation}
\begin{tikzcd}
\cB_1 & \cA_1 \arrow{l}[swap]{G_1} \\
\cB_0 \arrow{u}{P} & \cA_0 \arrow{l}[swap]{G_0} \arrow{u}[swap]{Q}
\end{tikzcd}
\end{equation}
be a commutative square of $\infty$-functors, the commutativity being witnessed by an equivalence
\[ u \colon P \circ G_0 \to G_1 \circ Q \]
Suppose that $G_0$ and $G_1$ admit left adjoints $F_0$ and $F_1$, respectively. Choose moreover unit transformations $\eta_1 \colon \mathrm{Id}_{\cA_1} \to G_1 F_1$, $\eta_0 \colon \mathrm{Id}_{\cA_0} \to G_0 F_0$ and counit transformations $\varepsilon_1 \colon F_1 G_1 \to \mathrm{Id}_{\cB_1}$, $\varepsilon_0 \colon F_0 G_0 \to \mathrm{Id}_{\cB_0}$.
Define a morphism $\mathrm{BC}(\sigma) \colon F_1 \circ P \to Q \circ F_0$ as the following composition:
\[ \begin{tikzcd}
F_1 P \arrow{r}{F_1 P \eta_0} & F_1 P G_0 F_0 \arrow{r}{F_1 u F_0} & F_1 G_1 Q F_0 \arrow{r}{\varepsilon_1 Q F_0} & Q F_0
\end{tikzcd} \]
We will refer to $\gamma$ as the \emph{Beck-Chevalley transformation associated to the square $\sigma$}.

\begin{defin}
	We will say that the original square \emph{is left adjointable} (or that \emph{it satisfies the Beck-Chevalley condition}) if $G_0$ and $G_1$ admit left adjoints and the Beck-Chevalley transformation $\mathrm{BC}(\sigma) \colon F_0 P \to QF$ is an equivalence.
\end{defin}

\begin{prop} \label{prop:left_adjointable_preserves_unit}
	Suppose that the given square is left adjointable.
	Then $P \eta_0 \simeq \eta_1 P$.
\end{prop}

\begin{proof}
	$\eta_1 P$ corresponds under the adjunction $F_1 \dashv G_1$ to the identity $F_1 P \to F_1 P$. We claim that
	\[ \begin{tikzcd}
	P \arrow{r}{P \eta_0} & P G_0 F_0 \arrow{r}{u F_0} & G_1 Q F_0 \arrow{r}{G_1 \gamma\inv} & G_1 F_1 P
	\end{tikzcd} \]
	is equivalent to $\eta_1 P$. Since $u$ is an equivalence, this produces an explicit equivalence between $P \eta_0$ and $\eta_1 P$, and therefore the proposition will be proved.
	In order to prove this, it will be sufficient to show that the composition
	\[ \begin{tikzcd}
	P \arrow{r}{P \eta_0} & P G_0 F_0 \arrow{r}{u F_0} & G_1 Q F_0
	\end{tikzcd} \]
	corresponds under the adjunction $F_1 \dashv G_1$ to $\gamma \colon F_1 P \to Q F_0$. However, the morphism $F_1 P \to Q F_0$ corresponding to the above composition under the adjuction is explicitly given by the composition
	\[ \begin{tikzcd}
	F_1 P \arrow{r}{F_1 P \eta} & F_1 P G_0 F_0 \arrow{r}{F_1 u F_0} & F_1 G_1 Q F_0 \arrow{r}{\varepsilon_0 Q F_0} & Q F_0
	\end{tikzcd} \]
	which we recognize being precisely the definition of $\gamma$.
	The proof is therefore complete.
\end{proof}

Let us go back to the original situation \eqref{eq:Beck_Chevalley_situation}.
We can interpret the diagram $\sigma$ as an $\infty$-functor $(\Delta^1 \times \Delta^1)^{\mathrm{op}} \to \Cat_\infty$, and therefore as a Cartesian fibration $q \colon \fX \to \Delta^1 \times \Delta^1$.
Here, we represent $\Delta^1 \times \Delta^1$ as
\begin{equation} \label{eq:square_representation}
\begin{tikzcd}
Y_1 \arrow{r}{a_{h,1}} \arrow{d}[swap]{a_{v,1}} & X_1 \arrow{d}{a_{v,0}} \\
Y_0 \arrow{r}{a_{h,0}} & X_0
\end{tikzcd}
\end{equation}
where the subscripts $v$ and $h$ stand for respectively ``vertical'' and ``horizontal''.

It is sometimes useful to characterize the condition of being left adjointable in terms of this Cartesian fibration.
This will be achieved in the next proposition (we learned this idea from \cite[Remark 4.1.5]{Lurie_Ambidexterity}).

\begin{prop} \label{prop:Beck_Chevalley_fibration_formulation}
	Keeping the above notations, the square \eqref{eq:Beck_Chevalley_situation} is left adjointable if and only if for every commutative square in $\fX$
	\[ \begin{tikzcd}
	\overline{Y_1} \arrow{r}{f'} \arrow{d}{g'} & \overline{X_1} \arrow{d}{g} \\
	\overline{Y_0} \arrow{r}{f} & \overline{X_0}
	\end{tikzcd} \]
	with $\overline{Y_i}$ (resp.\ $\overline{X_i}$) lying over $Y_i$ (resp.\ $X_i$), if $g$ and $g'$ are $q$-Cartesian and $f$ is locally $q$-coCartesian, then $f'$ is locally $q$-coCartesian as well.
\end{prop}

\begin{proof}
	Suppose given such a commutative square.
	Since $g$ is $q$-Cartesian, the morphism $f'$ is uniquely determined (up to coherent homotopy) by the composition $f \circ g' \colon \overline{Y_1} \to \overline{X_0}$.
	Let $\Delta^1 \to \Delta^1 \times \Delta^1$ be the morphism classifying the arrow $X_1 \to Y_1$. The pullback
	\[ \fX \times_{\Delta^1 \times \Delta^1} \Delta^1 \to \Delta^1 \]
	is a Cartesian fibration classifying $G_1 \colon \cA_1 \to \cB_1$.
	By hypothesis, $G_1$ has a left adjoint, hence this functor is also a coCartesian fibration.
	Therefore we can choose a local coCartesian lift $f'' \colon \overline{Y_1} \to \widetilde{X_1}$.
	The morphism $f'$ induces a well defined (up to homotopy) $\gamma \colon \widetilde{X_1} \to \overline{X_1}$ making the triangle in $\cC$
	\[ \begin{tikzcd}
	{} & \widetilde{X_1} \arrow{dr}{\gamma} \\
	\overline{Y_1} \arrow{ur}{f''} \arrow{rr}{f'} & & \overline{X_1}
	\end{tikzcd} \]
	commutative.
	The uniqueness of $f'$ forces $\gamma$ to be homotopic to the evaluation of $\mathrm{BC}(\sigma)$ on $\overline{X_0}$.
	At this point, it is sufficient to observe that $f'$ is locally $q$-coCartesian if and only if $\gamma$ is an equivalence, i.e.\ if and only if the original square satisfied the Beck-Chevalley condition.
\end{proof}

\begin{defin}
	We will say that a Cartesian fibration $p \colon \fX \to \Delta^1 \times \Delta^1$ \emph{satisfies the Beck-Chevalley condition} if the base changes along $\alpha_{h,0}, \alpha_{h,1} \colon \Delta^1\to \Delta^1 \times \Delta^1$ are also coCartesian and the condition of \cref{prop:Beck_Chevalley_fibration_formulation} is satisfied.
\end{defin}

\ifarxiv
\subsection{Stability under compositions}

The reformulation of the Beck-Chevalley condition provided by \cref{prop:Beck_Chevalley_fibration_formulation} allows us to prove a number of basic results which would require lengthy arguments otherwise. In this subsection we give a first example of the usefulness of such characterization by proving the stability of left adjointable squares under horizontal and vertical compositions.

\fi

\begin{prop} \label{prop:horizontal_composition_left_adjointable}
	Suppose given a commutative diagram in $\Cat_\infty$
	\[ \begin{tikzcd}
	\cC_1 & \cB_1 \arrow{l}[swap]{G_1'} & \cA_1 \arrow{l}[swap]{G_1} \\
	\cC_0 \arrow{u}[swap]{R} & \cB_0 \arrow{l}[swap]{G_0'} \arrow{u}[swap]{Q} & \cA_0 \arrow{l}[swap]{G_0} \arrow{u}[swap]{P}
	\end{tikzcd} \]
	If both squares are left adjointable, then the same goes for the outer one.
\end{prop}

\begin{proof}
	Let $K$ be the $1$-category depicted as
	\[ \begin{tikzcd}
	Z_1 \arrow{r} \arrow{d} & Y_1 \arrow{r} \arrow{d} & X_1 \arrow{d} \\
	Z_0 \arrow{r} & Y_0 \arrow{r} & X_0
	\end{tikzcd} \]
	There are three functors $\mathrm{sq}_r, \mathrm{sq}_l, \mathrm{sq}_{\mathrm{out}} \colon \Delta^1 \times \Delta^1 \to K$ selecting respectively the right square, the left square and the outer one.
	We can represent the original diagram as a Cartesian fibration $q \colon \fX \to K$.
	We let
	\[ q_r \colon \fX_r \to \Delta^1 \times \Delta^1, \qquad q_l \colon \fX_l \to \Delta^1 \times \Delta^1, \qquad q_{\mathrm{out}} \colon \fX_{\mathrm{out}} \to \Delta^1 \times \Delta^1 \]
	be the pullbacks of $q$ along $\mathrm{sq}_r$, $\mathrm{sq}_l$ and $\mathrm{sq}_{\mathrm{out}}$ (respectively).
	Observe that the induced morphisms $\fX_r \to \fX$, $\fX_l \to \fX$ and $\fX_{\mathrm{out}} \to \fX$ are fully faithful inclusions.
	
	We know that $q_r$ and $q_l$ satisfy the condition of \cref{prop:Beck_Chevalley_fibration_formulation} and we want to show that $q_{\mathrm{out}}$ has the same property.
	To do so, we consider a commutative diagram in $\fX$
	\[ \begin{tikzcd}
	\overline{Z_1} \arrow{r}{h'}  \arrow{d}{\alpha''} & \overline{X_1} \arrow{d}{\alpha} \\
	\overline{Z_0} \arrow{r}{h} & \overline{X_0}
	\end{tikzcd} \]
	where $\overline{Z_i}$ (resp.\ $\overline{X_i}$) lies over $Z_i$ (resp.\ $X_i$).
	Suppose furthermore that $\alpha$ and $\alpha''$ are $q$-Cartesian and that $h$ is locally $q$-coCartesian.
	Choose a locally $q$-coCartesian morphism $g \colon \overline{Z_0} \to \overline{Y_0}$ lying over $Z_0 \to Y_0$ and let $\alpha' \colon \overline{Y_1} \to \overline{Y_0}$ be a $q$-Cartesian morphism lying over $Y_1 \to Y_0$.
	Consider the diagram
	\[ \begin{tikzcd}
	\overline{Z_1} \arrow{dd}[swap]{\alpha''} \arrow[dashed]{dr}{g'} \arrow[bend left = 10]{drrr}{h'} \\
	{} & \overline{Y_1} \arrow[dashed]{rr}{f'} & & \overline{X_1} \arrow{dd}{\alpha} \\
	\overline{Z_0} \arrow{dr}{g} \arrow[bend left = 10]{drrr}{h} \\
	{} & \overline{Y_0} \arrow[crossing over, leftarrow]{uu}{\alpha'} \arrow[dashed]{rr}{f} & & \overline{X_0}
	\end{tikzcd} \]
	The universal property of $\alpha'$ allows to construct the morphism $g' \colon \overline{Z_1} \to \overline{Y_1}$, while the universal property of $h$ allows to construct $f \colon \overline{Y_1} \to \overline{X_0}$.
	\personal{$h$ is only locally $q$-coCartesian, but if we restrict $\fX$ along $\Delta^2 \to K$ (the bottom triangle), then $h$ becomes coCartesian, and this is the universal property we are referring to.}
	At this point, we can further use the universal property of $\alpha$ to construct $f' \colon \overline{Y_1} \to \overline{X_1}$.
	Since the left and the right square of the original diagram were left adjointable, we can use \cref{prop:Beck_Chevalley_fibration_formulation} to conclude that $f'$ and $g'$ are locally $q$-coCartesian.
	Consider the morphism $s \colon \Delta^2 \to K$ selecting the commutative triangle
	\[ \begin{tikzcd}
	{} & Y_1 \arrow{dr} \\
	Z_1 \arrow{ur} \arrow{rr} & & X_1
	\end{tikzcd} \]
	The pullback $\fX_s \coloneqq \fX \times_{K} \Delta^2 \to \Delta^2$ is both a Cartesian and a coCartesian fibration.
	Therefore, coCartesian morphisms of $\fX_s$ are stable under composition.
	It follows that $h \simeq f' \circ g'$ is locally $q$-coCartesian, completing the proof.
\end{proof}

\ifarxiv

\begin{prop}
	Suppose given a diagram in $\Cat_\infty$
	\[ \begin{tikzcd}
	\cB_2 & \cA_2 \arrow{l}[swap]{G_2} \\
	\cB_1 \arrow{u} & \cA_1 \arrow{l}[swap]{G_1} \arrow{u} \\
	\cB_0 \arrow{u} & \cA_0 \arrow{l}[swap]{G_0} \arrow{u}
	\end{tikzcd} \]
	If both the square are left adjointable, then so is the outer one.
\end{prop}

\begin{proof}
	Let $K$ be the $1$-category depicted as
	\[ \begin{tikzcd}
	Y_2 \arrow{r} \arrow{d} & X_2 \arrow{d} \\
	Y_1 \arrow{r} \arrow{d} & X_1 \arrow{d} \\
	Y_0 \arrow{r} & X_0
	\end{tikzcd} \]
	We review the original diagram as a Cartesian fibration $q \colon \fX \to K$.
	As in the proof of \cref{prop:horizontal_composition_left_adjointable}, we associate to $q$ three Cartesian fibrations $q_u \colon \fX_u \to \Delta^1 \times \Delta^1$, $q_d \colon \fX_d \to \Delta^1 \times \Delta^1$ and $q_{\mathrm{out}} \colon \fX_{\mathrm{out}} \to \Delta^1 \times \Delta^1$ respectively corresponding to the upper square, the lower square and the outer one.
	We are going to prove that $q_{\mathrm{out}}$ satisfies the condition of \cref{prop:Beck_Chevalley_fibration_formulation}.
	Let us therefore consider a commutative diagram in $\fX$
	\[ \begin{tikzcd}
	\overline{Y_2} \arrow{d}{h'} \arrow{r}{\alpha_2} & \overline{X_2} \arrow{d}{h} \\
	\overline{Y_0} \arrow{r}{\alpha_0} & \overline{X_0} \\
	\end{tikzcd} \]
	where $\overline{Y_i}$ (resp.\ $\overline{X_i}$) lies over $Y_i$ (resp.\ $X_i$), $h$ and $h'$ are $q$-Cartesian and $\alpha_0$ is $q$-coCartesian.
	We want to prove that $\alpha_2$ is locally $q$-coCartesian as well.
	Choose $q$-Cartesian morphisms $f' \colon \overline{Y_1} \to \overline{Y_0}$ and $f \colon \overline{X_1} \to \overline{X_0}$ lying over $Y_1 \to Y_0$ and $X_1 \to X_0$ respectively.
	Consider the diagram
	\[ \begin{tikzcd}
	\overline{Y_2} \arrow{rr}{\alpha_2} \arrow[dashed]{dr}{g'} \arrow[bend right = 10]{dddr}[swap]{h'} & & \overline{X_2} \arrow[dashed]{dr}{g} \\
	{} & \overline{Y_1} \arrow{dd}{f'} \arrow[dashed]{rr}{\alpha_1} & & \overline{X_1} \arrow{dd}{f} \\
	\\
	{} & \overline{Y_0} \arrow{rr}{\alpha_0} & & \overline{X_0} \arrow[leftarrow, crossing over, bend left = 10]{uuul}[near start]{h}
	\end{tikzcd} \]
	Using the universal property of $f'$, we construct $f' \colon \overline{Y_2} \to \overline{Y_1}$.
	Next, using twice the universal property of $f$, we also construct $\alpha_1 \colon \overline{Y_1} \to \overline{X_1}$ and $g \colon \overline{X_2} \to \overline{X_1}$.
	Since the lower square in the original diagram is left adjointable, \cref{prop:Beck_Chevalley_fibration_formulation} shows that $\alpha_1$ is locally $q$-coCartesian.
	As the upper square was left adjointable as well, we similarly conclude that $\alpha_2$ is locally $q$-coCartesian.
	The proof is therefore complete.
\end{proof}

\fi

\ifarxiv
\subsection{Stabilization of left adjointable squares}

We now turn to another application of \cref{prop:Beck_Chevalley_fibration_formulation} which is extremely useful in the economy of this article.
Namely, we will prove that the stabilization studied in \cite[§6.2.2]{Lurie_Higher_algebra} preserves left adjointable squares.
We will start by quickly reviewing this operation.

\fi

Fix a simplicial set $S$ and an inner fibration $p \colon \cC \to S$.
We define a functor $F \colon \sSet_{/S} \to \rSet$ by setting
\[ F(K) \coloneqq \Hom_S(K \times \cS^{\mathrm{fin}}_*, \cC) \]
where we review $K \times \cS^{\mathrm{fin}}_*$ as an element in $\sSet_{/S}$ via the projection
\[ K \times \cS^{\mathrm{fin}}_* \to K \to S. \]
This functor clearly commutes with colimits and it is therefore representable by a map
\[ q \colon \mathrm{PStab}(p) \to S \]
We will refer to $\PStab(p)$ as the \emph{prestabilization of $p$}.
It follows from the definition that a point of $\PStab(p)$ is a pair $(s, X)$ where $s \in S$ is a vertex and $X \colon \cS^{\mathrm{fin}}_* \to \cC_s$ is a functor.

Let us now suppose that the fibers $\cC_s$ of the inner fibration $p$ are pointed $\infty$-category with finite limits.
We let $\PStab(p)_*$ be the full sub-simplicial set of $\PStab(p)$ spanned by those points $(s, X)$ such that $X$ is a reduced functor.
We further let $\Stab(p)$ be the full sub-simplicial set of $\PStab_*(p)$ spanned by those points $(s, X)$ such that $X$ is a (strongly) excisive functor (that is, a spectrum object of $\cC_s$). We will refer to $\PStab_*(p)$ as the \emph{reduced prestabilization of $p$} and to $\Stab(p)$ as the \emph{stabilization of $p$}.

\cite[6.2.2.5]{Lurie_Higher_algebra} shows that the morphisms $q \colon \PStab(p) \to S$, $q' \colon \PStab_*(p) \to S$ and $q'' \colon \Stab(p) \to S$ are inner fibrations.
We now focus to the case $S = \Delta^1 \times \Delta^1$ and prove moreover that the Beck-Chevalley condition is stable under stabilization if certain additional mild hypotheses are satisfied:

\begin{prop} \label{prop:stabilization_of_left_adjointable_squares}
	Suppose given a Cartesian fibration $p \colon \cC \to \Delta^1 \times \Delta^1$ satisfying the Beck-Chevalley condition.
	Assume furthermore that the fibers $\cC_s$ of $q$ have finite limits.
	Then:
	\begin{enumerate}
		\item $\mathrm{PStab}(p)$ satisfies the Beck-Chevalley condition;
		\item if moreover the vertical functors (see the diagram \eqref{eq:square_representation}) are left exact and commute with sequential colimits, then $\mathrm{Stab}(p)$ satisfies the Beck-Chevalley condition.
	\end{enumerate}
\end{prop}

\begin{proof}
	It follows from \cite[6.2.2.13.(1)]{Lurie_Higher_algebra} that $q \colon \PStab(p) \to \Delta^1 \times \Delta^1$ is a Cartesian fibration.\footnote{We warn the reader that there are a couple of typos in the statement exchanging the word ``coCartesian'' with the word ``Cartesian'', but this is what it is proven.}
	Consider the morphisms $a_{h,0}, a_{h,1} \colon \Delta^1 \to \Delta^1 \times \Delta^1$.
	The two base changes $\cC \times_{\Delta^1 \times \Delta^1} \Delta^1 \to \Delta^1$ were coCartesian fibrations to start with, and the edges $a_{h,0}$ and $a_{h,1}$ were represented by left adjoints. Therefore we can combine \cite[6.2.2.3, 6.2.2.8.(1)]{Lurie_Higher_algebra} to deduce that $\PStab(p)$ is locally coCartesian over the edges $a_{h,0}$ and $a_{h,1}$.
	To complete the proof of point (1), we are left to show that the Beck-Chevalley condition holds for $\PStab(p)$.
	Consider therefore a commutative diagram in $\PStab(p)$
	\[ \begin{tikzcd}
	(Y_1, S_1) \arrow{r}{f'} \arrow{d}{g'} & (X_1, T_1) \arrow{d}{g} \\
	(Y_0, S_0) \arrow{r}{f} & (X_0, T_0)
	\end{tikzcd} \]
	where $Y_i$ and $X_i$ denote the vertexes of $\Delta^1 \times \Delta^1$ (accordingly to the diagram \eqref{eq:square_representation}) and $S_i$, $T_i$ are functors from $\cS^{\mathrm{fin}}_*$ to the fibers of $p$.
	Suppose furthermore that $g$ and $g'$ are $q$-Cartesian and that $f$ is locally $q$-coCartesian. We have to show that $f'$ is locally $q$-coCartesian as well.
	Unraveling the definitions of the $q$-Cartesian and locally $q$-coCartesian morphisms of $\PStab(p)$ given in \cite[6.2.2.8, 6.2.2.13]{Lurie_Higher_algebra}, we are reduced to fix a finite pointed space $K \in \cS^{\mathrm{fin}}_*$ and consider the induced commutative square in $\cC$
	\[ \begin{tikzcd}
	S_1(K) \arrow{r}{f'_K} \arrow{d}{g'_K} & T_1(K) \arrow{d}{g_K} \\
	S_0(K) \arrow{r}{f_K} & T_0(K)
	\end{tikzcd} \]
	We know that $g_K$ and $g'_K$ are $p$-Cartesian and that $f_K$ is locally $p$-coCartesian.
	Therefore, since $p$ was satisfying the Beck-Chevalley condition, we see that $f'_K$ is locally $p$-coCartesian for every $K \in \cS^{\mathrm{fin}}_*$.
	We conclude that $q \colon \PStab(p) \to \Delta^1 \times \Delta^1$ satisfies as well the Beck-Chevalley condition.
	
	We now turn to the second point.
	The hypotheses of \cite[6.2.2.13.(3)]{Lurie_Higher_algebra} are satisfied and therefore $q'' \colon \Stab(p) \to \Delta^1 \times \Delta^1$ is a Cartesian fibration.
	
	Consider again a commutative diagram in $\Stab(p)$
	\[ \begin{tikzcd}
	(Y_1, S_1) \arrow{r}{f'} \arrow{d}{g'} & (X_1, T_1) \arrow{d}{g} \\
	(Y_0, S_0) \arrow{r}{f} & (X_0, T_0)
	\end{tikzcd} \]
	where now $S_i$ and $T_i$ are spectrum objects in the fibers of $p$, $g$ and $g'$ are $q''$-Cartesian and $f$ is locally $q''$-coCartesian.
	Choose a locally $q$-coCartesian lift $\overline{f} \colon (Y_0, S_0) \to (X_0, \overline{T_0})$ in $\PStab(p)$ lying over $a_{h,0} \colon Y_0 \to X_0$.
	There is a natural morphism in $\PStab(p)$
	\[ \eta_0 \colon (X_0, \overline{T_0}) \to (X_0, T_0) \]
	and since $f$ was $q''$-coCartesian, we deduce from the proof of \cite[6.2.2.8.(3)]{Lurie_Higher_algebra} that this morphism can be identified with the unit of the adjunction whose right adjoint is the inclusion $\mathrm{Exc}_*(\cS^{\mathrm{fin}}_*, \cC_{X_0}) \hookrightarrow \Fun_*(\cS^{\mathrm{fin}}_*, \cC_{X_0})$.
	Let $g'' \colon (X_1, \overline{T_1}) \to (X_0, \overline{T_0})$ be a $q$-Cartesian morphism lying over $a_{v_0} \colon X_0 \to X_1$.
	We obtain a commutative diagram in $\PStab(p)$:
	\[ \begin{tikzcd}
	(Y_1, S_1) \arrow{dd}[swap]{g'} \arrow[dashed]{dr}[swap]{\overline{f'}} \arrow[bend left = 10]{drrr}{f'} \\
	{} & (X_1, \overline{T_1}) \arrow[dashed]{rr}{\eta_1} & & (X_1, T_1) \arrow{dd}{g} \\
	(Y_0, S_0) \arrow{dr}[swap]{\overline{f}} \arrow[bend left = 10]{drrr}{f} \\
	{} & (X_0, \overline{T_0}) \arrow[crossing over, leftarrow]{uu}[near end]{g''} \arrow{rr}{\eta_0} & & (X_0, T_0)
	\end{tikzcd} \]
	Since $g''$ is $q$-Cartesian, we are able to construct the morphism $\overline{f'} \colon (Y_1, S_1) \to (X_1, \overline{T_1})$ in $\PStab(p)$.
	Point (1) of this proposition shows then that $\overline{f'}$ is locally $q$-coCartesian.
	Since $g$ was $q''$-Cartesian it is $q$-Cartesian too (in virtue of \cite[6.2.2.13.(3)]{Lurie_Higher_algebra}).
	We obtain in this way the morphism $\eta_1 \colon (X_1, \overline{T_1}) \to (X_1, T_1)$.
	
	The proof of \cite[6.2.2.8.(3)]{Lurie_Higher_algebra} shows then that it is enough to prove that $\eta_1$ exhibits $T_1$ as the reflection of $\overline{T_1} \in \Fun_*(\cS^{\mathrm{fin}}_*, \cC_{X_1})$ in $\mathrm{Exc}_*(\cS^{\mathrm{fin}}_*, \cC_{X_1})$.
	This follows from the explicit construction of this reflection functor (see \cite[6.1.1.22, 6.1.1.27]{Lurie_Higher_algebra}, but also Remark 6.1.1.28 loc.\ cit.) and the fact that the functor $a_{v,0}\inv \colon \cC_{X_0} \to \cC_{X_1}$ commutes both with finite limits and with sequential colimits.
\end{proof}

\bibliographystyle{plain}
\bibliography{dahema}

\end{document}